\documentclass{amsart}

\usepackage{mathtools}
\mathtoolsset{showonlyrefs=true}
\usepackage{amssymb}
\usepackage{bm}
\usepackage[left=30mm,right=30mm,top=30mm,bottom=30mm]{geometry}
\usepackage{amsthm}

\newtheorem{lemma}{Lemma}[section]

\newtheorem{theorem}[lemma]{Theorem}

\theoremstyle{remark}
\newtheorem{example}[lemma]{Example}

\newcommand{\argmin}{\operatorname*{argmin}}
\newcommand{\cond}{\operatorname{cond}}
\newcommand{\diag}{\operatorname{diag}}
\newcommand{\Int}{\operatorname{Int}}

\renewcommand{\Re}{\operatorname{Re}}

\numberwithin{equation}{section}

\allowdisplaybreaks

\title[DSM-QR: derivation and its mathematical analysis]{Well-conditioned dipole-type method of fundamental solutions: derivation and its mathematical analysis}
\author{Koya Sakakibara}
\address[K.~Sakakibara]{Faculty of Mathematics and Physics, Institute of Science and Engineering, Kanazawa University, Kakuma-machi, Kanazawa-shi, Ishikawa 920-1192, Japan; RIKEN iTHEMS, 2-1 Hirosawa, Wako-shi, Saitama 351-0198, Japan}
\email{ksakaki@se.kanazawa-u.ac.jp}

\begin{document}

\begin{abstract}
    In this paper, we examine the dipole-type method of fundamental solutions, which can be conceptualized as a discretization of the "singularity-removed" double-layer potential. 
    We present a method for removing the ill-conditionality, which was previously considered a significant challenge, and provide a mathematical analysis in the context of disk regions. 
    Moreover, we extend the proposed method to the general Jordan region using conformal mapping, demonstrating the efficacy of the proposed method through numerical experiments.
\end{abstract}
\subjclass{65N80, 65N35, 65N12, 35J05, 35J08}
\keywords{Method of fundamental solutions, Double-layer potential, QR decomposition, Ill-conditionality}
\maketitle

\tableofcontents

\section{Introduction}

The method of fundamental solutions (MFS, for short) is a mesh-free numerical solver for partial differential equations, and it has been actively used especially in the field of engineering.
Its idea is very simple.
Let us consider the boundary-value problem for the linear partial differential operator $\mathcal{L}$; more precisely,
\begin{equation}
    \begin{dcases*}
        \mathcal{L}u=0&in $\Omega$,\\
        \mathcal{B}g=f&on $\partial\Omega$,
    \end{dcases*}
    \label{eq:model}
\end{equation}
where $\Omega\subset\mathbb{R}^2$ denotes a region such that $\mathbb{R}^2\setminus\overline{\Omega}\neq\emptyset$.
The second equation represents the linear boundary condition such as Dirichlet ($\mathcal{B}=\operatorname{id}$), Neumann ($\mathcal{B}=\partial_\nu$), or Robin ($\mathcal{B}=\operatorname{id}+\gamma\partial_\nu$).
Suppose that the partial differential operator $\mathcal{L}$ has a fundamental solution $E$; that is, $E$ satisfies
\begin{equation}
    \mathcal{L}E=\delta
\end{equation}
in the distributional sense.
The MFS offers an approximate solution to the problem~\eqref{eq:model} as in the following procedure.
\begin{enumerate}
    \item Choose $N\in\mathbb{N}$, and arrange the singular points $\{y_k\}_{k=1}^N\subset\mathbb{R}^2\setminus\overline{\Omega}$ ``suitably''.
    \item Seek an approximate solution $u^{(N)}$ of the form
    \begin{equation}
        u^{(N)}(x)=\sum_{k=1}^NQ_kE(x-y_k),\quad x\in\mathbb{R}^2\setminus\{y_1,y_2,\ldots,y_N\}.
    \end{equation}
    \item Determine the coefficients $\{Q_k\}_{k=1}^N$ by the collocation method; that is, choose $P\in\mathbb{N}$ so that $P\ge N$, arrange the collocation points $\{x_j\}_{j=1}^P\subset\partial\Omega$ ``suitably'', and impose the collocation equations:
    \begin{equation}
        u^{(N)}(x_j)=f(x_j),\quad
        j=1,2,\ldots,P,
    \end{equation}
    which is equivalent to the linear system
    \begin{equation}
        \mathbf{G}\bm{Q}=\bm{f},
        \label{eq:col_eq-linear}
    \end{equation}
    where
    \begin{equation}
        \mathbf{G}\coloneqq(\mathcal{B}E(x_j-y_k))_{\substack{j=1,2,\ldots,P\\k=1,2,\ldots,N}}\in\mathbb{R}^{P\times N},\quad
        \bm{Q}\coloneqq(Q_k)_{k=1}^N\in\mathbb{R}^N,\quad
        \bm{f}\coloneqq(f(x_j))_{j=1}^P\in\mathbb{R}^P.
    \end{equation}
\end{enumerate}

As previously stated, in contrast to the finite element method and the finite difference method, the MFS does not necessitate the partitioning of the region into a mesh. 
An approximate solution can be derived by simply positioning points in a suitable manner on the boundary and exterior of the region. 
Consequently, the MFS can be regarded as a mesh-free numerical method.
It has been demonstrated in numerous prior studies that when the singular and collocation points are correctly positioned and the boundary data is real-analytic, the approximation error exhibits a noteworthy property of decaying exponentially with respect to the number of points~\cite{katsurada1988mathematical,katsurada1989mathematical,katsurada1990asymptotic,katsurada1994charge,murota1995comparison,katsurada1996collocation,li2005convergence,smyrlis2006method,barnett2008stability,li2008convergence,li2009method,kazashi2014error,sakakibara2017asymptotic,sakakibara2017method,ei2022method}.
Nevertheless, the outcomes of these mathematical investigations remain inadequate. For instance, there is a paucity of findings pertaining to the existence of approximate solutions in multiply-connected regions in the plane, error estimation, and analysis in high-dimensional spaces.
Furthermore, the actual numerical computation with the MFS is subject to the issue of ill-conditionality. 
This refers to the fact that the condition number of the coefficient matrix $\mathbf{G}$ of a linear system~\eqref{eq:col_eq-linear} increases exponentially with respect to the number of points, indicating that the linear system~\eqref{eq:col_eq-linear} is exceedingly challenging to solve. 
Numerous prior studies have devised MFS-specific pre-processing techniques based on numerical linear algebra (see, for instance, \cite{branden2005preconditioners,branden2007discrete}). 
Nevertheless, the issue of ill-conditionality persists, and it has been identified as a significant impediment in actual numerical computation.

In light of these considerations, Antunes put forth an alternative approach in his paper~\cite{antunes2018reducing}, namely, modifying the basis function. 
This concept was previously proposed in \cite{fornberg2007stable} within the context of the radial basis function method.

The infinite series expansion obtained by the Taylor series expansion of the basis function is truncated by a finite number of terms (e.g., up to the $M$-th order terms) to obtain
\begin{equation}
    \Theta(r,\theta)=\mathbf{B}\mathbf{D}\mathbf{F}(r,\theta),
\end{equation}
where the singular points are arranged as $y_k=\varepsilon^{-1}(\cos\phi_k,\sin\phi_k)^\top$ and 
\begin{align}
    \mathbf{B}&\coloneqq\begin{pmatrix}
        1&\sin\phi_1&\cos\phi_1&\cdots&\sin M\phi_1&\cos M\phi_1\\
        1&\sin\phi_2&\cos\phi_2&\cdots&\sin M\phi_2&\cos M\phi_2\\
        \vdots&\vdots&\vdots&&\vdots&\vdots\\
        1&\sin\phi_N&\cos\phi_N&\cdots&\sin M\phi_N&\cos M\phi_N
    \end{pmatrix}\in\mathbb{R}^{N\times(2M+1)},\\
    \mathbf{D}&\coloneqq\diag\left(
        \log\varepsilon,\varepsilon,\varepsilon,\ldots,\frac{\varepsilon^M}{M},\frac{\varepsilon^M}{M}
    \right)\in\mathbb{R}^{(2M+1)\times(2M+1)},\\
    \bm{F}(r,\theta)&\coloneqq\begin{pmatrix}
        1\\
        r\sin\theta\\
        r\cos\theta\\
        \vdots\\
        r^M\cos M\theta\\
        r^M\sin M\theta
    \end{pmatrix}\in\mathbb{R}^{2M+1},
\end{align}
and $M$ is chosen to satisfy $2M+1>N$.
In this matrix representation, matrix $\mathbf{B}$ represents a well condition, whereas the ill-conditionality mainly stems from matrix $\mathbf{D}$. 
Accordingly, the matrix $\mathbf{B}$ is decomposed via the QR algorithm to yield the representation $\mathbf{B}=\mathbf{Q}\mathbf{R}$. 
The $N$-th leading principal submatrix of $\mathbf{D}$ is subsequently designated as $\mathbf{D}_N$. 
The new basis function is then expressed as
\begin{equation}
    \bm{\Psi}(r,\theta)
    =(\psi_k(r,\theta))_{k=1}^N
    \coloneqq\mathbf{D}_N^{-1}\mathbf{Q}^\top\bm{\Theta}(r,\theta)
    =\mathbf{D}_N^{-1}\mathbf{R}\mathbf{D}\bm{F}(r,\theta).
\end{equation}
Subsequently, he put forth the proposition of MFS-QR as a MFS-based system, founded upon this novel basis function. 
The approximate solution is expressed in the form of
\begin{equation}
    u^{(N)}(x)=u^{(N)}(r,\theta)=\sum_{k=1}^NQ_k\psi_k(r,\theta),\qquad
    x=\begin{pmatrix}
        r\cos\theta\\
        r\sin\theta
    \end{pmatrix}.
\end{equation}
The coefficients of the linear combination are determined by the collocation method. 

In paper~\cite{antunes2018reducing}, it was numerically shown that the ill-conditionality of the coefficient matrix can be completely eliminated for the disk domain, and that for other domains, the ill-conditionality cannot be completely eliminated, but can be reduced to some extent. 
Later, in paper~\cite{antunes2022well-conditioned}, a method based on singular value decomposition (MFS-SVD) instead of QR decomposition was developed, and it was numerically confirmed that the ill-conditionality can be removed when the domain is star-shaped. 
It should be noted that in both cases, the order of convergence is the same as that of the original MFS. 
In other words, MFS-QR and MFS-SVD achieve the exponential decay of errors, which is one of the major advantages of MFS, while eliminating one of the major problems of MFS, i.e., the problem of ill-conditionality. 
Therefore, it is expected that MFS-QR and MFS-SVD, rather than ordinary MFS, will become more important numerical methods in the future, but unfortunately, there are no theoretical guarantees to support the numerical experimental results obtained in previous studies.

It is also important to consider whether the construction of the approximate solution of the MFS is a sensible approach in the first place. 
In a sense, the approximate solution obtained through the MFS can be regarded as a "singularity-removed" discretization of the single-layer potential. 
However, as is well established in the theory of potential problems~\cite{folland1995introduction}, the single-layer potential should be employed for problems with Neumann boundary conditions, whereas the double-layer potential is more appropriate for problems with Dirichlet boundary conditions. 
In light of this awareness of the problem, the dipole-type method of fundamental solutions (also known as the dipole simulation method) was proposed. 
For further details, please refer to \S~\ref{sec:DSM}.
Henceforth, the dipole-type method of fundamental solutions will be referred to as the DSM.  
In the DSM, the approximate solution is constructed using the normal derivative of the fundamental solution as the basis function in lieu of the fundamental solution. 
In the original theoretical analysis of the MFS, it was necessary to make some mathematically unnatural assumptions. 
However, with the DSM, a similar theoretical analysis can be performed without the need to make unnatural assumptions. 
Therefore, it is more natural to use the DSM when dealing with Dirichlet boundary conditions. Nevertheless, even in the DSM, the problem of ill-conditionality persists, and to the best of the author's knowledge, there is no known solution to this problem.

The objective of this paper is to put forth a novel DSM, the DSM-QR, which is founded upon the tenets of the MFS-QR and eliminates the problematic ill-conditionality. 
Additionally, this paper endeavors to conduct a fundamental mathematical analysis of the DSM-QR. 
With the intention of establishing a robust foundation for future mathematical analysis, this paper concentrates on the Dirichlet boundary value problem
\begin{equation}
    \begin{dcases*}
        \triangle u=0&in $\mathbb{B}_\rho$,\\
        u=f&on $\partial\mathbb{B}_\rho$,
    \end{dcases*}
    \label{eq:BVP}
\end{equation}
for the Laplace equation in the disk domain, where $\mathbb{B}_\rho$ denotes the disk in the complex plane with radius $\rho$ having the origin as its center; that is,
\begin{equation}
    \mathbb{B}_\rho\coloneqq\{z\in\mathbb{C}\mid|z|<\rho\}.
\end{equation}
The collocation points $\{x_j\}_{j=1}^P$ and singular points $\{y_k\}_{k=1}^N$ are to be placed as
\begin{alignat}{2}
    x_j&=\rho\exp\left(\frac{2\pi\mathrm{i}j}{P}\right),&\qquad&j=1,2,\ldots,P,\\
    y_k&=R\exp\left(\frac{2\pi\mathrm{i}k}{N}\right),&\qquad&k=1,2,\ldots,N,
\end{alignat}
where $R>\rho$ and $\mathbb{R}^2$ and $\mathbb{C}$ are identified here and hereafter.
Moreover, set
\begin{equation}
    \kappa\coloneqq\frac{\rho}{R}.
\end{equation}

The following three theorems represent the core of this paper's theoretical framework.

\begin{theorem}
    \label{thm:unique-existence}
    The DSM-QR gives a unique approximate solution.
\end{theorem}

Denote by $\cond_2(\mathbf{A})$ the condition number of $\mathbf{A}$ induced by the $\ell^2$-norm.

\begin{theorem}
    \label{thm:condition-number}
    The condition number $\cond_2(\mathbf{G})$ of the coefficient matrix $\mathbf{G}$ is of $\mathrm{O}(1)$ as $N\to\infty$.
    More precisely, it holds that
    \begin{equation}
        \cond_2(\mathbf{G})=\frac{2}{1+\kappa^{N-2}},\quad
        N\in\mathbb{N}.
    \end{equation}
\end{theorem}

\begin{theorem}
    \label{thm:error_estimate}
    \begin{enumerate}
        \item Suppose that the Fourier series $\sum_{n\in\mathbb{Z}}f_n\mathrm{e}^{\mathrm{i}n\theta}$ of $\theta\mapsto f(\rho\mathrm{e}^{\mathrm{i}\theta})$ is absolutely convergent.
        Then, $u^{(N)}$ uniformly converges to $u$ on $\mathbb{B}_\rho$; that is,
        \begin{equation}
            \|u-u^{(N)}\|_{L^\infty(\mathbb{B})}\longrightarrow0\quad(N\to\infty).
        \end{equation}
        \item Suppose there exists an $\alpha>1$ such that $|f_n|=\mathrm{O}(|n|^{-\alpha})$ ($|n|\to\infty$) holds.
        Then, there exists a positive constant $C$, independent of $N$, such that
        \begin{equation}
            \|u-u^{(N)}\|_{L^\infty(\mathbb{B}_\rho)}\le CN^{-\alpha+1}
        \end{equation}
        holds.
        \item Suppose that there exists an $a\in(0,1)$ such that $|f_n|=\mathrm{O}(a^{|n|})$ holds as $|n|\to\infty$.
        Then, there exists a positive constant $C$, independent of $N$, such that
        \begin{equation}
            \|u-u^{(N)}\|_{L^\infty(\mathbb{B}_\rho)}\le C\times\begin{dcases*}
                a^{N/2}&if $a>\kappa^2$,\\
                N\kappa^N&if $a=\kappa^2$,\\
                \kappa^N&if $a<\kappa^2$.
            \end{dcases*}
        \end{equation}
    \end{enumerate}
\end{theorem}

The following is a description of the organization of this paper. 
In Section~\ref{sec:DSM}, we undertake a review of the MFS from the perspective of potential theory, and subsequently derive a tabular expression for the approximate solution by the DSM. 
In Section~\ref{sec:DSM-QR}, we put forth the DSM-QR, which is based on the concept of the MFS-QR. 
Section~\ref{sec:unique-existence} presents a mathematical analysis of DSM-QR. 
First, we provide a detailed account of the basis functions utilized in the DSM-QR through the presentation of concrete calculations pertaining to QR decomposition. 
Subsequently, the eigenvalues of the coefficient matrix are obtained in Lemma~\ref{lem:G}, which also demonstrates that the condition number of the coefficient matrix is of $\mathrm{O}(1)$. 
Then, we provide a proof of the error estimate presented in Theorem~\ref{thm:error_estimate}. 
In Section~\ref{sec:numerics}, we present the results of numerical experiments. 
First, we consider the case of the disk region and demonstrate that the problematic ill-conditionality is entirely eliminated and that the error decays exponentially, as predicted by the main theorem of this paper. 
Next, we extend the DSM-QR to the case of Jordan regions, although no mathematical analysis is available. 
Specifically, DSM-QR is extended to the general Jordan region by mapping point configurations in the disk region to the Jordan region by a conformal mapping, and by viewing basis functions as perturbations of basis functions in the disk region. 
The utility of this approach is demonstrated through numerical experiments. 
Section~\ref{sec:summary} provides a summary of the contents of this paper and suggests avenues for future research.

\section{Dipole simulation method}
\label{sec:DSM}

In this section, we briefly explain the idea of the dipole simulation method (DSM).

Let $\Gamma$ be a Jordan curve surrounding $\Omega$; more precisely, $\overline{\Omega}\subset\Int\Omega[\Gamma]$, where $\Omega[\Gamma]$ denotes a Jordan region with boundary $\Gamma$ and $\Int\Omega[\Gamma]$ denotes the interior of $\Omega[\Gamma]$.
Suppose that there exists a function $Q$ defined on $\Gamma$ such that the solution $u$ to the Dirichlet boundary value problem~\eqref{eq:BVP} can be expressed as
\begin{equation}
    u(z)=\int_\Gamma E(z-\zeta)Q(\zeta)\,\mathrm{d}\sigma(\zeta)\eqqcolon\mathrm{SL}[Q](z),
\end{equation}
which can regarded as a ``singularity-removed'' single-layer potential.
Here, $E=E(z)=-(2\pi)^{-1}\log|z|$ is the fundamental solution to the Laplace operator $-\triangle$.
If we consider a linear combination of Dirac delta distributions $Q^{(N)}(z)\coloneqq\sum_{k=1}^NQ_k\delta(\zeta-\zeta_k)$ with $\zeta_k\in\Gamma$ ($k=1,2,\ldots,N$), we formally obtain that
\begin{equation}
    \mathrm{SL}[Q^{(N)}](z)
    =\sum_{k=1}^NQ_kE(z-\zeta_k),
\end{equation}
which is an approximate solution by the MFS.
Namely, the MFS can be regarded as a discretization of the ``singularity-removed'' single-layer potential representation.

However, it is natural to consider the double-layer potential, not single-layer potential, when we deal with the Dirichlet boundary condition (see, for instance, \cite{folland1995introduction}).
Recall that the dipole-layer potential is given by replacing the integral kernel $E(z-\zeta)$ with its normal derivative as
\begin{equation}
    \mathrm{DL}[Q](z)\coloneqq\int_\Gamma\frac{\partial E}{\partial\nu(\zeta)}(z-\zeta)\,\mathrm{d}\zeta,
\end{equation}
where $\nu(\zeta)$ denotes the unit outward normal vector to $\Gamma$.
Then, we formally obtain
\begin{equation}
    \mathrm{DL}[Q^{(N)}](z)
    =\sum_{k=1}^NQ_k\frac{\partial E}{\partial\nu(\zeta_k)}(z-\zeta_k).
    \label{eq:DLP}
\end{equation}
Therefore, it is natural to consider that a function of the form~\eqref{eq:DLP} can also approximate the solution to the Dirichlet boundary value problem~\eqref{eq:BVP}.
This is the idea of the DSM.
More precisely, the algorithm of the DSM can be described as follows:
\begin{enumerate}
    \item Choose $N\in\mathbb{N}$, and arrange the singular points $\{\zeta_k\}_{k=1}^N\subset\mathbb{C}\setminus\overline{\Omega}$ and the dipole moments $\{\nu_k\}_{k=1}^N\subset S^1$ suitably, where $S^1\coloneqq\{z\in\mathbb{C}\mid|z|=1\}$ denotes the unit circle.
    \item Seek an approximate solution $u^{(N)}$ of the form
    \begin{equation}
        u^{(N)}(z)=\sum_{k=1}^NQ_k\frac{\partial E}{\partial\nu_k}(z-\zeta_k).
    \end{equation}
    \item Determine the coefficients $\{Q_k\}_{k=1}^N$ by the collocation method; that is, arrange the collocation points $\{z_j\}_{j=1}^P\subset\partial\Omega$ suitably, and impose the collocation equations:
    \begin{equation}
        u^{(N)}(z_j)=f(z_j),\quad
        j=1,2,\ldots,P.
    \end{equation}
\end{enumerate}

\section{DSM-QR}
\label{sec:DSM-QR}

In this section, we develop DSM-QR algorithm.
In particular, we replace the basis functions $\{(\partial E/\partial\nu_k)(\cdot-\zeta_k)\}_{k=1}^N$ with new ones by a similar idea to the MFS-QR.

Let us arrange the singular points $\zeta_k$ on the circle with radius $R$ uniformly as
\begin{equation}
    \zeta_k\coloneqq R\omega^k\quad(k=1,2,\ldots,N),\quad
    \omega\coloneqq\exp\frac{2\pi\mathrm{i}}{N}.
\end{equation}
Regarding the circle $|z|=R$ as $\Gamma$, it is natural to define dipole moments $\nu_k$ as
\begin{equation}
    \nu_k\coloneqq\frac{\zeta_k}{|\zeta_k|}=\omega^k\quad(k=1,2,\ldots,N).
\end{equation}
Then, the basis function of the DSM can be computed as
\begin{equation}
    \frac{\partial E}{\partial\nu_k}(z-\zeta_k)
    =-\frac{1}{2\pi}\Re\left(\frac{\nu_k}{z-\zeta_k}\right)
    =-\frac{1}{2\pi}\Re\left(\frac{\omega^k}{z-R\omega^k}\right).
    \label{eq:DSM-basis}
\end{equation}

For $z=r\mathrm{e}^{\mathrm{i}\theta}$ and $\zeta=R\mathrm{e}^{\mathrm{i}\phi}$ with $0\le r\le\rho<R$, we formally have
\begin{align}
    &-\frac{1}{2\pi}\Re\left(\frac{\zeta/|\zeta|}{z-\zeta}\right)\\
    &=-\frac{1}{2\pi}\Re\left(\frac{\mathrm{e}^{\mathrm{i}\phi}}{r\mathrm{e}^{\mathrm{i}\theta}-R\mathrm{e}^{\mathrm{i}\phi}}\right)\\
    &=\frac{1}{2\pi R}\Re\left[\sum_{n=0}^\infty\left(\frac{r}{R}\right)^n\mathrm{e}^{\mathrm{i}n(\theta-\phi)}\right]\\
    &=\frac{1}{2\pi R}\left[1+\sum_{n=1}^\infty\left(\frac{r}{R}\right)^n\cos n(\theta-\phi)\right]\\
    &=\frac{1}{2\pi R}\left[1+\sum_{n=1}^\infty\left(\frac{r}{R}\right)^n\left(\cos n\theta\cos n\phi+\sin n\theta\sin n\phi\right)\right]\\
    &=\frac{1}{2\pi R}\begin{pmatrix}
        1&\sin\phi&\cos\phi&\sin2\phi&\cos2\phi&\cdots
        \end{pmatrix}
        \begin{pmatrix}
            1\\
            &\kappa\\
            &&\kappa\\
            &&&\kappa^2\\
            &&&&\kappa^2\\
            &&&&&\ddots
        \end{pmatrix}\begin{pmatrix}
            1\\
            s\sin\theta\\
            s\cos\theta\\
            s^2\sin2\theta\\
            s^2\cos2\theta\\
            \vdots
        \end{pmatrix},
\end{align}
where $\kappa\coloneqq\rho/R\in(0,1)$ and $s\coloneqq r/\rho\in[0,1]$.
Basis functions~\eqref{eq:DSM-basis} of the DSM can be obtained by taking $\rho=\phi_k\coloneqq2\pi k/N$ ($k=1,2,\ldots,N$).
Namely, 
\begin{multline}
    \begin{pmatrix}
        \dfrac{\partial E}{\partial\nu_1}(z-\zeta_1)\\[2ex]
        \dfrac{\partial E}{\partial\nu_2}(z-\zeta_2)\\[2ex]
        \vdots\\
        \dfrac{\partial E}{\partial\nu_N}(z-\zeta_N)
    \end{pmatrix}
    =
    \frac{1}{2\pi R}\begin{pmatrix}
        1&\sin\phi_1&\cos\phi_1&\sin2\phi_1&\cos2\phi_1&\cdots\\
        1&\sin\phi_2&\cos\phi_2&\sin2\phi_2&\cos2\phi_2&\cdots\\
        \vdots&\vdots&\vdots&\vdots&\vdots\\
        1&\sin\phi_N&\cos\phi_N&\sin2\phi_N&\cos2\phi_N&\cdots
    \end{pmatrix}\\[-3ex]
    \begin{pmatrix}
        1\\
        &\kappa\\
        &&\kappa\\
        &&&\kappa^2\\
        &&&&\kappa^2\\
        &&&&&\ddots
    \end{pmatrix}\begin{pmatrix}
        1\\
        s\sin\theta\\
        s\cos\theta\\
        s^2\sin2\theta\\
        s^2\cos2\theta\\
        \vdots
    \end{pmatrix}
\end{multline}
Hereafter, we omit the constant factor $1/(2\pi R)$.

Let $M\in\mathbb{N}$ satisfying $2M+1>N$, and truncate the above infinite series expansion at the $M$-th order.
More precisely, set
\begin{align}
    \mathbf{B}&\coloneqq\begin{pmatrix}
        1&\sin\phi_1&\cos\phi_1&\cdots&\sin M\phi_1&\cos M\phi_1\\
        1&\sin\phi_2&\cos\phi_2&\cdots&\sin M\phi_2&\cos M\phi_2\\
        \vdots&\vdots&\vdots&&\vdots&\vdots\\
        1&\sin\phi_N&\cos\phi_N&\cdots&\sin M\phi_N&\cos M\phi_N
    \end{pmatrix}\in\mathbb{R}^{N\times(2M+1)},\\
    \mathbf{D}&\coloneqq\diag\left(1,\kappa,\kappa,\ldots,\kappa^M,\kappa^M\right)\in\mathbb{R}^{(2M+1)\times(2M+1)},\\
    \bm{F}(s,\theta)&\coloneqq\begin{pmatrix}
        1\\
        s\sin\theta\\
        s\cos\theta\\
        \vdots\\
        s^M\sin M\theta\\
        s^M\cos M\theta
    \end{pmatrix}\in\mathbb{R}^{2M+1}.
\end{align}
We perform the QR decomposition of $\mathbf{B}$ as $\mathbf{B}=\mathbf{Q}\mathbf{R}$, where $\mathbf{Q}\in\mathbb{R}^{N\times N}$ is an orthogonal matrix and $\mathbf{R}\in\mathbb{R}^{N\times(2M+1)}$ is an upper triangular matrix.
Furthermore, let $\mathbf{D}_N$ be the $N$-th leading principal submatrix of $\mathbf{D}$.
Then, we define new basis functions as
\begin{equation}
    \bm{\Psi}(s,\theta)
    =(\psi_k(s,\theta))_{k=1}^N
    \coloneqq\mathbf{D}_N^{-1}\mathbf{Q}^\top\mathbf{Q}\mathbf{R}\mathbf{D}\bm{F}(s,\theta)
    =\mathbf{D}_N^{-1}\mathbf{R}\mathbf{D}\bm{F}(s,\theta).
    \label{eq:DSM-QR_basis}
\end{equation}

Hence, the DSM-QR gives an approximate solution to problem~\eqref{eq:BVP} of the form
\begin{equation}
    u^{(N)}(z)
    =u^{(N)}(r,\theta)
    =\sum_{k=1}^NQ_k\psi_k(r/\rho,\theta),\quad
    z=r\mathrm{e}^{\mathrm{i}\theta}.
\end{equation}
Coefficients $\{Q_k\}_{k=1}^N$ are determined by the collocation method; that is,
\begin{equation}
    u^{(N)}(\rho\mathrm{e}^{\mathrm{i}\theta_j})=f(\rho\mathrm{e}^{\mathrm{i}\theta_j}),\quad
    \theta_j\coloneqq\frac{2\pi j}{P},\quad
    j=1,2,\ldots,P,
\end{equation}
which is equivalent to the linear system
\begin{equation}
    \mathbf{G}\bm{Q}=\bm{f},
    \label{eq:DSM-QR_collocaiton-equation-matrix}
\end{equation}
where
\begin{equation}
    \mathbf{G}\coloneqq(\psi_k(1,\theta_j))_{\substack{j=1,2,\ldots,P\\k=1,2,\ldots,N}}\in\mathbb{R}^{P\times N},
    \quad\bm{Q}\coloneqq(Q_k)_{k=1}^N\in\mathbb{R}^N,\quad
    \bm{f}\coloneqq(f(\rho\mathrm{e}^{\mathrm{i}\theta_j}))_{j=1}^P\in\mathbb{R}^P.
\end{equation}

\section{Mathematical analysis of DSM-QR}
\label{sec:unique-existence}

\subsection{Concrete expression of basis functions of DSM-QR}

Hereafter, we choose $M=N-1$ and suppose that $N$ is odd and greater than $1$ for the sake of simplicity.
We explicitly compute the QR decomposition of $\mathbf{B}$.
To this end, define vectors $\bm{c}_j,\bm{s}_j\in\mathbb{R}^N$ by
\begin{equation}
    \bm{c}_j
    \coloneqq
    \begin{pmatrix}
        \cos j\phi_1\\
        \cos j\phi_2\\
        \vdots\\
        \cos j\phi_N
    \end{pmatrix},
    \quad
    \bm{s}_j
    \coloneqq
    \begin{pmatrix}
        \sin j\phi_1\\
        \sin j\phi_2\\
        \vdots\\
        \sin j\phi_N
    \end{pmatrix}
\end{equation}
for $j\in\mathbb{Z}$.
Then, the matrix $\mathbf{B}$ can be represented as
\begin{equation}
    \mathbf{B}
    =\begin{pmatrix}
        \bm{c}_0
        &\bm{s}_1
        &\bm{c}_1
        &\cdots
        &\bm{s}_N
        &\bm{c}_N
    \end{pmatrix}.
\end{equation}

\begin{lemma}
    \label{lem:c_s}
    It holds that
    \begin{alignat}{2}
        &\langle\bm{c}_j,\bm{c}_{j'}\rangle_{\ell^2}=\langle\bm{s}_j,\bm{s}_{j'}\rangle_{\ell^2}=0&\quad&\text{for}\ 0\le j,j'\le\frac{N-1}{2}\ \text{with}\ j\neq j',\\
        &\langle\bm{c}_j,\bm{s}_{j'}\rangle_{\ell^2}=0&\quad&\text{for}\ 0\le j,j'\le\frac{N-1}{2},\\
        &\|\bm{c}_0\|_{\ell^2}=\sqrt{N},\\
        &\|\bm{c}_j\|_{\ell^2}=\|\bm{s}_j\|_{\ell^2}=\sqrt{\frac{N}{2}}&\quad&\text{for}\ 1\le j\le\frac{N-1}{2}.
    \end{alignat}
    Furthermore, it holds that
    \begin{equation}
        \bm{c}_{N-j}=\bm{c}_j,\quad
        \bm{s}_{N-j}=-\bm{s}_j,\quad
        1\le j\le\frac{N-1}{2}.
    \end{equation}
\end{lemma}

Since the above lemma immediately follows from the elementary relation
\begin{equation}
    \sum_{l=1}^N\omega^{jl}=\begin{dcases*}
        0&if $j\not\equiv0$,\\
        N&otherwise,
    \end{dcases*}
    \label{eq:omega_sum}
\end{equation}
we omit its proof.
Here and hereafter, the congruence $\equiv$ is taken in modulo $N$.

Define
\begin{align}
    \bm{q}_j
    &\coloneqq\begin{dcases*}
        \frac{\bm{c}_l}{|\bm{c}_l|}&if $j=2l$, $l=0,1,\ldots,\dfrac{N-1}{2}$,\\
        \frac{\bm{s}_l}{|\bm{s}_l|}&if $j=2l-1$, $l=1,2,\ldots,\dfrac{N-1}{2}$
    \end{dcases*}\\
    &=\begin{dcases*}
        \frac{\bm{c}_0}{\sqrt{N}}&if $j=0$,\\
        \frac{\bm{s}_l}{\sqrt{N/2}}&if $j=2l-1$, $l=1,2,\ldots,\dfrac{N-1}{2}$,\\
        \frac{\bm{c}_l}{\sqrt{N/2}}&if $j=2l$, $l=1,2,\ldots,\dfrac{N-1}{2}$.
    \end{dcases*}
\end{align}
Then, Lemma~\ref{lem:c_s} yields that
\begin{align}
    \mathbf{B}
    &=\begin{pmatrix}
        \bm{c}_0&\bm{s}_1&\bm{c}_1&\cdots&\bm{s}_{\frac{N-1}{2}}&\bm{c}_{\frac{N-1}{2}}&\bm{s}_{\frac{N+1}{2}}&\bm{c}_{\frac{N+1}{2}}&\cdots&\bm{s}_{N-1}&\bm{c}_{N-1}
    \end{pmatrix}\\
    &=\begin{pmatrix}
        \bm{c}_0&\bm{s}_1&\bm{c}_1&\cdots&\bm{s}_{\frac{N-1}{2}}&\bm{c}_{\frac{N-1}{2}}&-\bm{s}_{\frac{N-1}{2}}&\bm{c}_{\frac{N-1}{2}}&\cdots&-\bm{s}_1&\bm{c}_1
    \end{pmatrix}\\
    &=\begin{pmatrix}
        \bm{c}_0&\bm{s}_1&\bm{c}_1&\cdots&\bm{s}_{\frac{N-1}{2}}&\bm{c}_{\frac{N-1}{2}}
    \end{pmatrix}\begin{pmatrix}
        1\\
        &1&&&&&&&&-1\\
        &&1&&&&&&&&1\\
        &&&\ddots&&&&&\reflectbox{$\ddots$}\\
        &&&&1&&-1\\
        &&&&&1&&1
    \end{pmatrix}\\
    &=\begin{pmatrix}
        \bm{q}_0&\bm{q}_1&\bm{q}_2&\cdots&\bm{q}_{N-2}&\bm{q}_{N-1}
    \end{pmatrix}\sqrt{\frac{N}{2}}\begin{pmatrix}
        \sqrt{2}\\
        &1&&&&&&&&-1\\
        &&1&&&&&&&&1\\
        &&&\ddots&&&&&\reflectbox{$\ddots$}\\
        &&&&1&&-1\\
        &&&&&1&&1
    \end{pmatrix}.
\end{align}
Therefore, setting
\begin{align}
    \bm{Q}&\coloneqq\begin{pmatrix}
        \bm{q}_0&\bm{q}_1&\bm{q}_2&\cdots&\bm{q}_{N-2}&\bm{q}_{N-1}
    \end{pmatrix},\\
    \mathbf{R}&\coloneqq\sqrt{N}\begin{pmatrix}
        \sqrt{2}\\
        &1&&&&&&&&-1\\
        &&1&&&&&&&&1\\
        &&&\ddots&&&&&\reflectbox{$\ddots$}\\
        &&&&1&&-1\\
        &&&&&1&&1
    \end{pmatrix},
\end{align}
we have $\mathbf{B}=\mathbf{Q}\mathbf{R}$.
Since $\mathbf{Q}$ is an orthogonal matrix by Lemma~\ref{lem:c_s} and $\mathbf{R}$ is an upper triangular matrix, this is the desired QR decomposition of $\mathbf{B}$.
A direct computation shows that
\begin{align}
    &\mathbf{D}_N^{-1}\mathbf{R}\mathbf{D}=(\tilde{r}_{ij})_{\substack{i=1,2,\ldots,N\\j=1,2,\ldots,2N-1}},\\
    &\tilde{r}_{ij}=\begin{dcases*}
        \sqrt{N}&if $i=j=1$,\\
        \sqrt{\frac{N}{2}}&if $i=j=2,3,\ldots,N$,\\
        -\sqrt{\frac{N}{2}}\kappa^{-\lfloor i/2\rfloor+\lfloor j/2\rfloor}&if $i=2l\ \left(l=1,2,\ldots,\dfrac{N-1}{2}\right)$, $j=2N-i$,\\
        \sqrt{\frac{N}{2}}\kappa^{-\lfloor i/2\rfloor+\lfloor j/2\rfloor}&if $i=2l+1\ \left(l=1,2,\ldots,\dfrac{N-1}{2}\right)$, $j=2(N+1)-i$.
    \end{dcases*}
\end{align}
Therefore, the basis functions $\bm{\Psi}=(\psi_k)_{k=1}^N$ of the DSM-QR, defined as \eqref{eq:DSM-QR_basis}, can be explicitly given by
\begin{equation}
    \psi_k(s,\theta)=\begin{dcases*}
        \sqrt{N}&if $k=1$,\\
        \sqrt{\frac{N}{2}}\left[s^m\sin m\theta-\kappa^{N-2m}s^{N-m}\sin(N-m)\theta\right]&if $k=2m$, $m=1,2,\ldots,\dfrac{N-1}{2}$,\\
        \sqrt{\frac{N}{2}}\left[s^m\cos m\theta+\kappa^{N-2m}s^{N-m}\cos(N-m)\theta\right]&if $k=2m+1$, $m=1,2,\ldots,\dfrac{N-1}{2}$.
    \end{dcases*}
\end{equation}

\subsection{Unique existence and condition number}

In this subsection, we prove that the linear system~\eqref{eq:DSM-QR_collocaiton-equation-matrix} is uniquely solvable.
We will establish mathematical theory for $P=\alpha N$ with $\alpha\in\mathbb{N}$.
When $\alpha=1$, the linear system~\eqref{eq:DSM-QR_collocaiton-equation-matrix} is understood in a standard manner.
When $\alpha\ge2$, we seek a solution to the linear sytem~\eqref{eq:DSM-QR_collocaiton-equation-matrix} in the least-squares sense; that is,
\begin{equation}
    \bm{Q}\in\argmin_{\tilde{\bm{Q}}\in\mathbb{R}^N}\|\mathbf{G}\bm{Q}-\bm{f}\|_{\ell^2}^2,
\end{equation}
which is equivalent to
\begin{equation}
    \mathbf{G}^\top\mathbf{G}\bm{Q}=\mathbf{G}^\top\bm{f}.
    \label{eq:collocation_linear-system_least-squares}
\end{equation}
Hereafter, we focus on the case where $\alpha=1$.
Other cases $\alpha\ge2$ can be shown in a similar manner.

We examine properties of the coefficient matrix $\mathbf{G}^\top\mathbf{G}$.

\begin{lemma}
    \label{lem:G}
    The matrix $\mathbf{G}^\top\mathbf{G}\in\mathbb{R}^{N\times N}$ is diagonal and its diagonal elements are given by
    \begin{equation}
        \left[\mathbf{G}^\top\mathbf{G}\right]_{kk}=\begin{dcases*}
            N^2&if $k=1$,\\
            \frac{N^2}{4}(1+\kappa^{N-2m})^2&if $k=2m$ or $k=2m+1$ with $m=1,2,\ldots,\dfrac{N-1}{2}$.
        \end{dcases*}
    \end{equation}
\end{lemma}

\begin{proof}
    By the definition of $\mathbf{G}$, we have
    \begin{equation}
        \mathbf{G}^\top\mathbf{G}=\left(\sum_{j=1}^N\psi_k(1,\theta_j)\psi_l(1,\theta_j)\right)_{k,l}\in\mathbb{R}^{N\times N},
    \end{equation}
    and
    \begin{equation}
        \psi_k(1,\theta_j)=\begin{dcases*}
            \sqrt{N}&if $k=1$,\\
            \sqrt{\frac{N}{2}}\left[\sin m\theta_j-\kappa^{N-2m}\sin(N-m)\theta_j\right]&if $k=2m$, $m=1,2,\ldots,\dfrac{N-1}{2}$,\\
            \sqrt{\frac{N}{2}}\left[\cos m\theta_j+\kappa^{N-2m}\cos(N-m)\theta_j\right]&if $k=2m+1$, $m=1,2,\ldots,\dfrac{N-1}{2}$.
        \end{dcases*}
        \label{eq:psi_1}
    \end{equation}
    Using an elementary relation~\eqref{eq:omega_sum}, it can be easily shown that the following hold:
    \begin{align}
        &\sum_{j=1}^N\sin m\theta_j=0,\quad m\in\mathbb{Z},\label{eq:sum_sin}\\
        &\sum_{j=1}^N\cos m\theta_j=\begin{dcases*}
            N&if $m\equiv0$,\\
            0&otherwise,
        \end{dcases*}\label{eq:sum_cos}\\
        &\sum_{j=1}^N\cos m\theta_j\cos n\theta_j=\begin{dcases*}
            \frac{N}{2}&if $(m-n\equiv0)\veebar(m+n\equiv0)$,\\
            N&if $(m-n\equiv0)\land(m+n\equiv0)$,\\
            0&otherwise,
        \end{dcases*}\label{eq:sum_coscos}\\
        &\sum_{j=1}^N\sin m\theta_j\sin n\theta_j=\begin{dcases*}
            \frac{N}{2}&if $(m-n\equiv0)\land(m+n\not\equiv0)$,\\
            -\frac{N}{2}&if $(m+n\equiv0)\land(m-n\not\equiv0)$,\\
            0&otherwise,
        \end{dcases*}\label{eq:sum_sinsin}\\
        &\sum_{j=1}^N\sin m\theta_j\cos n\theta_j=0,\quad m,n\in\mathbb{Z},\label{eq:sum_sincos}
    \end{align}
    where the symbol $\veebar$ denotes the exclusive or.

    Off-diagonal elements can be computed as in the following manner.
    $(1,2n)$ and $(1,2n+1)$ elements for $n=1,2,\ldots,(N-1)/2$ are
    \begin{align}
        [\mathbf{G}^\top\mathbf{G}]_{1,2n}
        &=\frac{N}{\sqrt{2}}\sum_{j=1}^N\left[\sin n\theta_j-\kappa^{N-2n}\sin(N-n)\theta_j\right]=0,\\
        [\mathbf{G}^\top\mathbf{G}]_{1,2n+1}
        &=\frac{N}{\sqrt{2}}\sum_{j=1}^N\left[\cos n\theta_j+\kappa^{N-2n}\cos(N-n)\theta_j\right]=0
    \end{align}
    by \eqref{eq:sum_sin} and \eqref{eq:sum_cos}.
    $(2m,2n)$ elements for $m=1,2,\ldots,(N-1)/2$ and $n=m+1,\ldots,(N-1)/2$ are
    \begin{align}
        [\mathbf{G}^\top\mathbf{G}]_{2m,2n}
        &=\frac{N}{2}\sum_{j=1}^N\left[\sin m\theta_j\sin n\theta_j-\kappa^{N-2n}\sin m\theta_j\sin(N-n)\theta_j\right.\\
        &\hspace{50pt}\left.-\kappa^{N-2m}\sin(N-m)\theta_j\sin n\theta_j+\kappa^{2(N-(m+n))}\sin(N-m)\theta_j\sin(N-n)\theta_j\right]\\
        &=0
    \end{align}
    by \eqref{eq:sum_sinsin}.
    $(2m,2n+1)$ elements for $m=1,2,\ldots,(N-1)/2$ and $n=m,\ldots,(N-1)/2$ are
    \begin{align}
        [\mathbf{G}^\top\mathbf{G}]_{2m,2n+1}
        &=\frac{N}{2}\sum_{j=1}^N\left[\sin m\theta_j\cos n\theta_j+\kappa^{N-2n}\sin m\theta_j\cos(N-n)\theta_j\right.\\
        &\hspace{30pt}\left.-\kappa^{N-2m}\sin(N-m)\theta_j\cos n\theta_j-\kappa^{2(N-(m+n))}\sin(N-m)\theta_j\cos(N-n)\theta_j\right]\\
        &=0
    \end{align}
    by \eqref{eq:sum_sincos}.
    $(2m+1,2n+1)$ elements for $m=1,2,\ldots,(N-1)/2$ and $n=m+1,\ldots,(N-1)/2$ are
    \begin{align}
        [\mathbf{G}^\top\mathbf{G}]_{2m+1,2n+1}
        &=\frac{N}{2}\sum_{j=1}^N\left[\cos m\theta_j\cos n\theta_j+\kappa^{N-2n}\cos m\theta_j\cos(N-n)\theta_j\right.\\
        &\hspace{30pt}\left.+\kappa^{N-2m}\cos(N-m)\theta_j\cos n\theta_j+\kappa^{2(N-(m+n))}\cos(N-m)\theta_j\cos(N-n)\theta_j\right]
    \end{align}
    by \eqref{eq:sum_coscos}.
    In the same way in computing $(2m,2n+1)$ elements, we can show that $(2m+1,2n)$ elements are equal to $0$ for $m=1,2,\ldots,(N-1)/2$ and $n=m+1,\ldots,(N-1)/2$.
    Summarizing the above, we have shown that upper triangular elements of $\mathbf{G}^\top\mathbf{G}$ are equal to $0$.
    Since $\mathbf{G}^\top\mathbf{G}$ is symmetric, it follows that $\mathbf{G}^\top\mathbf{G}$ is diagonal.

    Concerning diagonal elements, we have by \eqref{eq:sum_coscos} and \eqref{eq:sum_sinsin} that
    \begin{align}
        [\mathbf{G}^\top\mathbf{G}]_{11}
        &=\sum_{j=1}^NN
        =N^2,\\
        [\mathbf{G}^\top\mathbf{G}]_{2m,2m}
        &=\frac{N}{2}\sum_{j=1}^N\left[\sin^2m\theta_j-2\kappa^{N-2m}\sin m\theta_j\sin(N-m)\theta_j+\kappa^{2(N-2m)}\sin^2(N-m)\theta_j\right]\\
        &=\frac{N^2}{4}(1+\kappa^{N-2m})^2,\\
        [\mathbf{G}^\top\mathbf{G}]_{2m+1,2m+1}
        &=\frac{N}{2}\sum_{j=1}^N\left[\cos^2m\theta_j+2\kappa^{N-2m}\cos m\theta_j\cos(N-m)\theta_j+\kappa^{2(N-2m)}\cos^2(N-m)\theta_j\right]\\
        &=\frac{N^2}{4}(1+\kappa^{N-2m})^2
    \end{align}
    for $m=1,2,\ldots,(N-1)/2$.
\end{proof}

The above lemma shows that $\mathbf{G}^\top\mathbf{G}$ is non-singular since all the eigenvalues of $\mathbf{G}^\top\mathbf{G}$ are positive.
Namely, Theorem~\ref{thm:unique-existence} has been proved.

\begin{proof}[Proof of Theorem~\ref{thm:condition-number}]
    It is well known that the relation
    \begin{equation}
        \cond_2(\mathbf{A})=\sqrt{\frac{\lambda_{\max}(\mathbf{A}^\top\mathbf{A})}{\lambda_{\min}(\mathbf{A}^\top\mathbf{A})}}
        \label{eq:cond_2}
    \end{equation}
    holds, where $\lambda_{\max}(\mathbf{A}^\top\mathbf{A})$ and $\lambda_{\min}(\mathbf{A}^\top\mathbf{A})$ are the maximal and minimal (by moduli) eigenvalues of $\mathbf{A}^\top\mathbf{A}$, respectively.
    Lemma~\ref{lem:G} shows that
    \begin{equation}
        \lambda_{\max}(\mathbf{G}^\top\mathbf{G})=N^2,\quad
        \lambda_{\min}(\mathbf{G}^\top\mathbf{G})=\frac{N^2}{4}(1+\kappa^{N-2})^2
    \end{equation}
    hold.
    Therefore, the relation~\eqref{eq:cond_2} yields that
    \begin{equation}
        \cond_2(\mathbf{G})
        =\sqrt{\frac{\lambda_{\max}(\mathbf{G}^\top\mathbf{G})}{\lambda_{\min}(\mathbf{G}^\top\mathbf{G})}}
        =\frac{2}{1+\kappa^{N-2}}
        \longrightarrow2\quad(N\to\infty).
    \end{equation}
\end{proof}

The above corollary implies that the coefficient matrix $\mathbf{G}$ in the collocation equation~\eqref{eq:DSM-QR_collocaiton-equation-matrix} is well-conditioned, which is a quite difference compared with the original DSM.

\subsection{Error estimate}
\label{sec:error-estimate}

Since $\mathbf{G}^\top\mathbf{G}$ is diagonal by Lemma~\ref{lem:G}, the solution to \eqref{eq:collocation_linear-system_least-squares} is explicitly given by
\begin{equation}
    Q_k=\frac{1}{[\mathbf{G}^\top\mathbf{G}]_{kk}}\sum_{j=1}^N\psi_k(1,\theta_j)f(\rho\mathrm{e}^{\mathrm{i}\theta_j}),\quad
    k=1,2,\ldots,N.
    \label{eq:Q}
\end{equation}
To derive the error estimate, we write down the exact solution $u$ to the problem~\eqref{eq:BVP} using the Fourier series expansion as
\begin{equation}
    u(r\mathrm{e}^{\mathrm{i}\theta})=\sum_{n\in\mathbb{Z}}f_n\left(\frac{r}{\rho}\right)^{|n|}\mathrm{e}^{\mathrm{i}n\theta},
\end{equation}
where $f_n$ denotes the $n$-th Fourier coefficient of the function $\theta\mapsto f(\rho\mathrm{e}^{\mathrm{i}\theta})$.
We next represent $u^{(N)}$ as an infinite series using $\{f_n\}_{n\in\mathbb{Z}}$.
The concrete expression~\eqref{eq:Q} of coefficients $\{Q_k\}_{k=1}^N$ yields that
\begin{align}
    u^{(N)}(r\mathrm{e}^{\mathrm{i}\theta})
    &=\sum_{k=1}^N\left[\frac{1}{[\mathbf{G}^\top\mathbf{G}]_k}\sum_{j=1}^N\psi_k(1,\theta_j)f(\rho\mathrm{e}^{\mathrm{i}\theta_j})\right]\psi_k(r/\rho,\theta)\\
    &=\sum_{k=1}^N\frac{1}{[\mathbf{G}^\top\mathbf{G}]_k}\sum_{j=1}^N\psi_k(1,\theta_j)\left(\sum_{n\in\mathbb{Z}}f_n\mathrm{e}^{\mathrm{i}n\theta_j}\right)\psi_k(r/\rho,\theta)\\
    &=\sum_{n\in\mathbb{Z}}f_n\sum_{k=1}^N\frac{1}{[\mathbf{G}^\top\mathbf{G}]_k}\left(\sum_{j=1}^N\psi_k(1,\theta_j)\mathrm{e}^{\mathrm{i}n\theta_j}\right)\psi_k(r/\rho,\theta).
\end{align}

\begin{lemma}
    \label{lem:phi}
    It holds that
    \begin{align}
        \varphi_n^{(N)}(r/\rho,\theta)
        &\coloneqq\sum_{k=1}^N\left(\frac{1}{[\mathbf{G}^\top\mathbf{G}]_k}\sum_{j=1}^N\psi_k(1,\theta_j)\mathrm{e}^{\mathrm{i}n\theta_j}\right)\psi_k(r/\rho,\theta)\\
        &=\begin{dcases*}
            1&if $n\equiv0$,\\
            \frac{1}{1+\kappa^{N-2m}}\left(w^m+\kappa^{N-2m}\overline{w}^{N-m}\right)&if $n\equiv m$,\\
            \frac{1}{1+\kappa^{N-2m}}\left(\overline{w}^m+\kappa^{N-2m}w^{N-m}\right)&if $n\equiv-m$
        \end{dcases*}
    \end{align}
    with $m=1,\ldots,(N-1)/2$, where we set $w\coloneqq(r/\rho)\mathrm{e}^{\mathrm{i}\theta}$.
\end{lemma}

\begin{proof}
    We first compute $\sum_{j=1}^N\psi_k(1,\theta_j)\mathrm{e}^{\mathrm{i}n\theta_j}$.
    The elementary relation~\eqref{eq:omega_sum} and expression~\eqref{eq:psi_1} yield that
    \begin{equation}
        \sum_{j=1}^N\psi_1(1,\theta_j)
        =\sqrt{N}\sum_{j=1}^N\mathrm{e}^{\mathrm{i}n\theta_j}
        =\begin{dcases*}
            N^{3/2}&if $n\equiv0$,\\
            0&otherwise
        \end{dcases*}
    \end{equation}
    for $k=1$,
    \begin{align}
        &\sum_{j=1}^N\psi_{2m}(1,\theta_j)\mathrm{e}^{\mathrm{i}n\theta_j}\\
        &=\sqrt{\frac{N}{2}}\sum_{j=1}^N\left(\frac{\mathrm{e}^{\mathrm{i}(m+n)\theta_j}-\mathrm{e}^{-\mathrm{i}(m-n)\theta_j}}{2\mathrm{i}}-\kappa^{N-2m}\frac{\mathrm{e}^{\mathrm{i}(N-(m-n))\theta_j}-\mathrm{e}^{-\mathrm{i}(N-(m+n))\theta_j}}{2\mathrm{i}}\right)\\
        &=\mathrm{i}\left(\frac{N}{2}\right)^{3/2}\times\begin{dcases*}
            1+\kappa^{N-2m}&if $n\equiv m$,\\
            -(1+\kappa^{N-2m})&if $n\equiv-m$,\\
            0&otherwise
        \end{dcases*}
    \end{align}
    for $k=2m$ with $m=1,\ldots,(N-1)/2$, and
    \begin{align}
        &\sum_{j=1}^N\psi_{2m+1}(1,\theta_j)\mathrm{e}^{\mathrm{i}n\theta_j}\\
        &=\sqrt{\frac{N}{2}}\sum_{j=1}^N\left(\frac{\mathrm{e}^{\mathrm{i}(m+n)\theta_j}+\mathrm{e}^{-\mathrm{i}(m-n)\theta_j}}{2}+\kappa^{N-2m}\frac{\mathrm{e}^{\mathrm{i}(N-(m-n))\theta_j}+\mathrm{e}^{-\mathrm{i}(N-(m+n))\theta_j}}{2}\right)\\
        &=\left(\frac{N}{2}\right)^{3/2}\times\begin{dcases*}
            1+\kappa^{N-2m}&if $(n\equiv m)\lor(n\equiv-m)$,\\
            0&otherwise
        \end{dcases*}
    \end{align}
    for $k=2m+1$ with $m=1,\ldots,(N-1)/2$.
    Therefore, it follows from the above result and Lemma~\ref{lem:G} that
    \begin{equation}
        \sum_{k=1}^N\frac{1}{[\mathbf{G}^\top\mathbf{G}]_k}\left(\sum_{j=1}^N\psi_k(1,\theta_j)\mathrm{e}^{\mathrm{i}n\theta_j}\right)\psi_k(r/\rho,\theta)
        =\frac{1}{[\mathbf{G}^\top\mathbf{G}]_1}\cdot N^{3/2}\psi_1(r/\rho,\theta)
        =1
    \end{equation}
    for $n\equiv0$, 
    \begin{align}
        &\sum_{k=1}^N\frac{1}{[\mathbf{G}^\top\mathbf{G}]_k}\left(\sum_{j=1}^N\psi_k(1,\theta_j)\mathrm{e}^{\mathrm{i}n\theta_j}\right)\psi_k(r/\rho,\theta)\\
        &=\left(\frac{N}{2}\right)^{3/2}(1+\kappa^{N-2m})\left(\frac{\mathrm{i}}{[\mathbf{G}^\top\mathbf{G}]_{2m}}\psi_{2m}(r/\rho,\theta)+\frac{1}{[\mathbf{G}^\top\mathbf{G}]_{2m+1}}\psi_{2m+1}(r/\rho,\theta)\right)\\
        &=\frac{1}{1+\kappa^{N-2m}}\left(w^m+\kappa^{N-2m}\overline{w}^{N-m}\right)
    \end{align}
    for $n\equiv m$, and
    \begin{align}
        &\sum_{k=1}^N\frac{1}{[\mathbf{G}^\top\mathbf{G}]_k}\left(\sum_{j=1}^N\psi_k(1,\theta_j)\mathrm{e}^{\mathrm{i}n\theta_j}\right)\psi_k(r/\rho,\theta)\\
        &=\left(\frac{N}{2}\right)^{3/2}(1+\kappa^{N-2m})\left(\frac{-\mathrm{i}}{[\mathbf{G}^\top\mathbf{G}]_{2m}}\psi_{2m}(r/\rho,\theta)+\frac{1}{[\mathbf{G}^\top\mathbf{G}]_{2m+1}}\psi_{2m+1}(r/\rho,\theta)\right)\\
        &=\frac{1}{1+\kappa^{N-2m}}\left(\overline{w}^m+\kappa^{N-2m}w^{N-m}\right)
    \end{align}
    for $n\equiv-m$.
\end{proof}

Using the functions $\varphi_n^{(N)}$ defined in Lemma~\ref{lem:phi}, we can represent $u^{(N)}(r\mathrm{e}^{\mathrm{i}\theta})$ as
\begin{equation}
    u^{(N)}(r\mathrm{e}^{\mathrm{i}\theta})=\sum_{n\in\mathbb{Z}}f_n\varphi_n^{(N)}(r/\rho,\theta).
\end{equation}
Since $u$ and $u^{(N)}$ are harmonic in $\mathbb{B}_\rho$ and continuous on $\overline{\mathbb{B}}_\rho$, the maximum principle gives the following bound:
\begin{align}
    \|u-u^{(N)}\|_{L^\infty(\mathbb{B}_\rho)}
    &\le\sum_{n\in\mathbb{Z}}|f_n|g_n^{(N)},
    \label{eq:error_bound}
\end{align}
where
\begin{equation}
    g_n^{(N)}\coloneqq\sup_{\theta\in\mathbb{R}}\left|\mathrm{e}^{\mathrm{i}n\theta}-\varphi_n^{(N)}(1,\theta)\right|.
\end{equation}

To derive the desired error estimate, the local and global estimates on $g_n^{(N)}$ are required.

\begin{lemma}
    \label{lem:g}
    \begin{enumerate}
        \item For any $n\in\mathbb{N}$, it holds that
        \begin{equation}
            g_n^{(N)}\le2.
        \end{equation}
        \item It holds that
        \begin{equation}
            g_0^{(0)}=0.
        \end{equation}
        Moreover, for $n=1,\ldots,(N-1)/2$, it holds that
        \begin{equation}
            g_n^{(N)}\le2\kappa^{N-2n}.
        \end{equation}
    \end{enumerate}
\end{lemma}

\begin{proof}
    (i)
    Lemma~\ref{lem:phi} yields that
    \begin{equation}
        \varphi_n^{(N)}(1,\theta)=\begin{dcases*}
            1&if $n\equiv0$,\\
            \frac{1}{1+\kappa^{N-2m}}\left(\mathrm{e}^{\mathrm{i}m\theta}+\kappa^{N-2m}\mathrm{e}^{-\mathrm{i}(N-m)\theta}\right)&if $n\equiv m$,\\
            \frac{1}{1+\kappa^{N-2m}}\left(\mathrm{e}^{-\mathrm{i}m\theta}+\kappa^{N-2m}\mathrm{e}^{\mathrm{i}(N-m)\theta}\right)&if $n\equiv-m$
        \end{dcases*}
    \end{equation}
    with $m=1,\ldots,(N-1)/2$.
    Therefore, the triangle inequality implies that $\varphi_n^{(N)}(1,\theta)\le1$ for any $n\in\mathbb{Z}$ and $\theta\in\mathbb{R}$.
    We thus obtain
    \begin{equation}
        g_n^{(N)}
        \le\sup_{\theta\in\mathbb{R}}\left(|\mathrm{e}^{\mathrm{i}n\theta}|+|\varphi_n^{(N)}(1,\theta)|\right)
        \le2.
    \end{equation}

    (ii)
    Since $\varphi_0^{(N)}(1,\theta)=1$ as shown in the above, we readily see that $g_0^{(N)}=0$ holds.
    For $n=1,\ldots,(N-1)/2$, we have
    \begin{align}
        \mathrm{e}^{\mathrm{i}n\theta}-\varphi_n^{(N)}(1,\theta)
        &=\mathrm{e}^{\mathrm{i}n\theta}-\frac{1}{1+\kappa^{N-2n}}\left(\mathrm{e}^{\mathrm{i}n\theta}+\kappa^{N-2n}\mathrm{e}^{-\mathrm{i}(N-n)\theta}\right)\\
        &=\frac{\kappa^{N-2n}}{1+\kappa^{N-2n}}\left(\mathrm{e}^{\mathrm{i}n\theta}-\mathrm{e}^{-\mathrm{i}(N-n)\theta}\right).
    \end{align}
    Hence, the triangle inequality yields that
    \begin{equation}
        \left|\mathrm{e}^{\mathrm{i}n\theta}-\varphi_n^{(N)}(1,\theta)\right|
        \le\frac{2\kappa^{N-2n}}{1+\kappa^{N-2n}}
        \le2\kappa^{N-2n},
    \end{equation}
    which is the desired estimate.
\end{proof}

\begin{proof}[Proof of Theorem~\ref{thm:error_estimate}]
    (i), (ii)
    Note that $\varphi_{-n}^{(N)}=\overline{\varphi_n^{(N)}}$ holds, which implies that $g_{-n}^{(N)}=g_n^{(N)}$.
    Based on this relation, we decompose the right-hand side of \eqref{eq:error_bound} as follows:
    \begin{equation}
        \sum_{n\in\mathbb{Z}}|f_n|g_n^{(N)}
        =|f_0|g_0^{(N)}+\sum_{n=1}^{\lfloor N/3\rfloor-1}(|f_n|+|f_{-n}|)g_n^{(N)}+\sum_{n=\lfloor N/3\rfloor}^\infty(|f_n|+|f_{-n}|)g_n^{(N)}
        \equiv\mathrm{I}+\mathrm{II}+\mathrm{III}.
    \end{equation}

    Since $g_0^{(N)}=0$ by Lemma~\ref{lem:g}(ii), we have $\mathrm{I}=0$.
    Concerning $\mathrm{II}$, since $|f_n|\le\|f\|_{L^\infty(\partial\mathbb{B}_\rho)}$ holds, Lemma~\ref{lem:g}(ii) implies that
    \begin{align}
        \mathrm{II}
        \le4\|f\|_{L^\infty(\partial\mathbb{B}_\rho)}\sum_{n=1}^{\lfloor N/3\rfloor-1}\kappa^{N-2n}
        \le4\|f\|_{L^\infty(\partial\mathbb{B}_\rho)}\kappa^N\cdot\frac{\kappa^{-2\lfloor N/3\rfloor}}{\kappa^{-2}-1}
        \le\frac{4\|f\|_{L^\infty(\partial\mathbb{B}_\rho)}}{\kappa^{-2}-1}\kappa^{\frac{N}{3}},
    \end{align}
    Lemma~\ref{lem:g}(i) yields that
    \begin{equation}
        \mathrm{III}
        \le2\sum_{n=\lfloor N/3\rfloor}^\infty(|f_n|+|f_{-n}|)
        =\begin{dcases*}
            \mathrm{o}(1)&for case (i),\\
            \mathrm{O}(N^{-\alpha+1})&for case (ii).
        \end{dcases*}
    \end{equation}

    Summarizing the above, we complete the proof of (i) and (ii).

    (iii)
    We decompose the right-hand side of \eqref{eq:error_bound} as
    \begin{equation}
        \sum_{n\in\mathbb{Z}}|f_n|g_n^{(N)}
        =|f_0|g_0^{(N)}+\sum_{n=1}^{(N-1)/2}(|f_n|+|f_{-n}|)g_n^{(N)}+\sum_{n=(N+1)/2}^\infty(|f_n|+|f_{-n}|)g_n^{(N)}
        \equiv\mathrm{I}+\mathrm{II}+\mathrm{III}.
    \end{equation}

    Since $g_0^{(0)}=0$ by Lemma~\ref{lem:g}(ii), we have $\mathrm{I}=0$.
    The assumption $|f_n|=\mathrm{O}(a^{|n|})$ ($|n|\to\infty$) assures the existence of positive constant $C$ satisfying $|f_n|\le Ca^{|n|}$ for all $n\in\mathbb{Z}$.
    Combining this relation with Lemma~\ref{lem:g}(ii), we have
    \begin{equation}
        \mathrm{II}
        \le4C\sum_{n=1}^{(N-1)/2}a^n\kappa^{N-2n}
        =4C\kappa^N\sum_{n=1}^{(N-1)/2}(a\kappa^{-2})^n
        \le
        4C\times\begin{dcases*}
            \frac{(a\kappa^{-2})^{1/2}}{a\kappa^{-2}-1}a^{N/2}&if $a>\kappa^2$,\\
            N\kappa^N&if $a=\kappa^2$,\\
            \frac{1}{1-a\kappa^{-2}}\kappa^N&if $a<\kappa^2$,
        \end{dcases*}
    \end{equation}
    where we have used
    \begin{equation}
        \sum_{n=1}^m\tau^n\le\begin{dcases*}
            \frac{\tau^{m+1}}{\tau-1}&if $\tau>1$,\\
            m&if $\tau=1$,\\
            \frac{1}{1-\tau}&if $0<\tau<1$.\\
        \end{dcases*}
    \end{equation}
    Concerning $\mathrm{III}$, the assumption $|f_n|=\mathrm{O}(a^{|n|})$ ($n\to\infty$) and Lemma~\ref{lem:g} imply that
    \begin{equation}
        \mathrm{III}
        \le4C\sum_{n=(N+1)/2}^\infty a^n
        =\frac{4C}{1-a}a^{(N+1)/2}.
    \end{equation}

    Combining the above estimates, we complete the proof of (iii).
\end{proof}

\section{Numerical experiments}
\label{sec:numerics}

In this section, we show several results of numerical experiments that compare the original MFS and MFS-QR.

\subsection{Disk}

We first consider the problem~\eqref{eq:BVP} with boundary value $f(x,y)=x^2y^3$.
We set parameters as $\rho\coloneqq1$ and $R\coloneqq1.1$ or $R\coloneqq1.5$.
We change $N$ as $N=2l+1$ with $l=1,2,\ldots,l_{\max}\coloneqq150$, and compute condition numbers and approximation errors.
\begin{figure}[tb]
    \begin{minipage}{.5\hsize}
        \centering
        \includegraphics[width=\hsize]{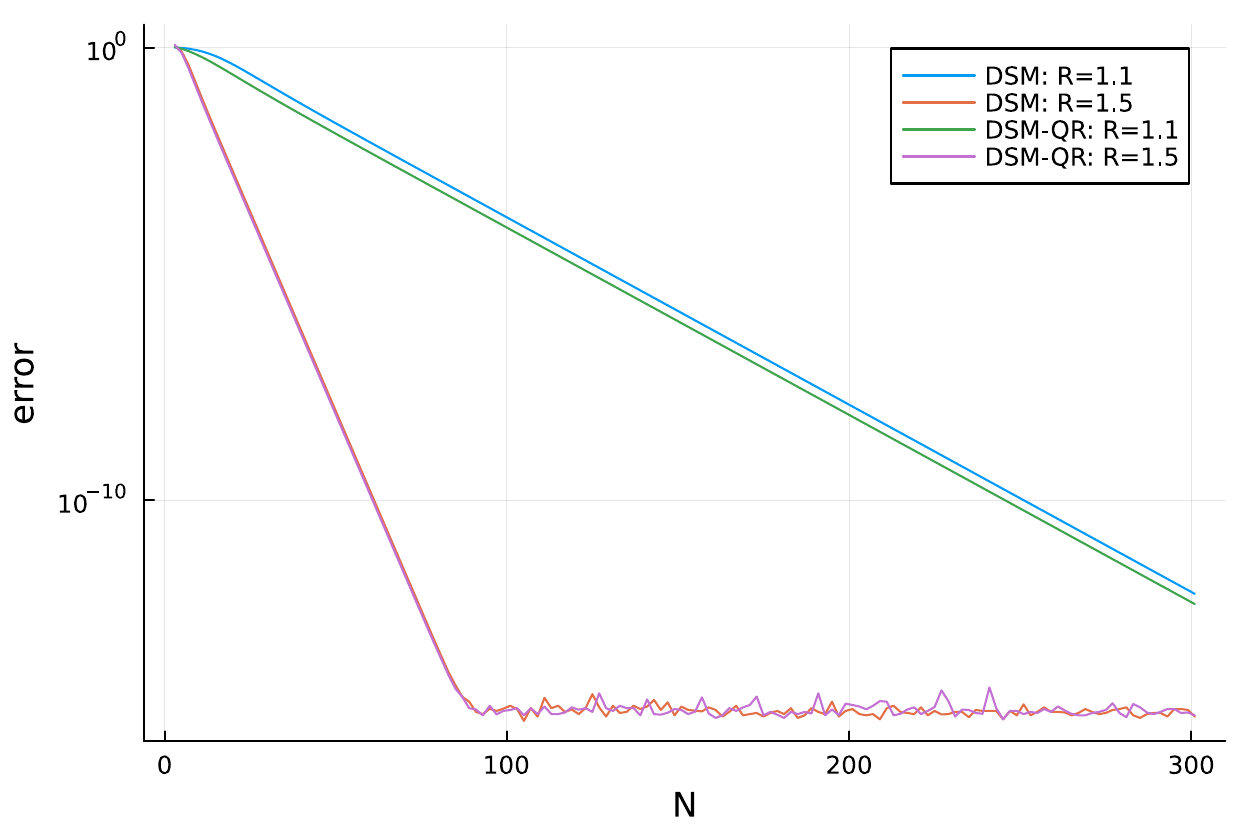}
    \end{minipage}
    \begin{minipage}{.5\hsize}
        \centering
        \includegraphics[width=\hsize]{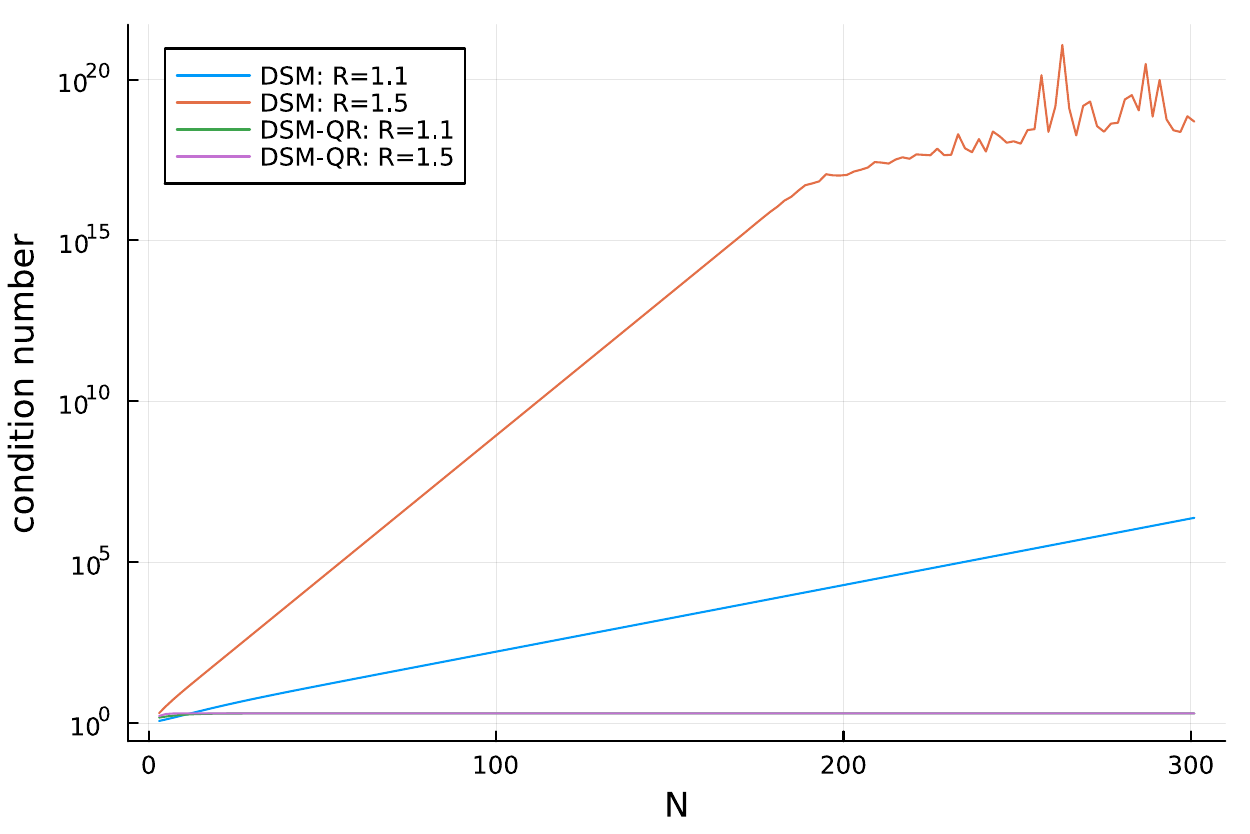}
    \end{minipage}%
    \caption{Plots of the $L^\infty$-norms of errors and the condition numbers.
    (Left) $L^\infty$-error; (Right) condition number when $\Omega$ is a disk region.}
    \label{fig:disk_error_cond-num}
\end{figure}
Figure~\ref{fig:disk_error_cond-num} shows the result of numerical experiments.
It can be observed that the $L^\infty$-norms of approximation errors decay exponentially with respect to $N$ for both the DSM and the DSM-QR.
The condition number increases exponentially for the DSM, while it is of $\mathrm{O}(1)$ for the DSM-QR.

\subsection{Jordan region}

We next deal with the case where $\Omega$ is a Jordan region.
In the previous study~\cite{antunes2018reducing}, it was pointed out that the MFS-QR reduces the ill-conditioning but cannot remove it completely when the problem region $\Omega$ is not a disk.
In this subsection, we develop an analogue of the MFS-QR for Jordan region with analytic boundary using conformal mappings.

Since $\Omega$ is a Jordan region, the Riemann mapping theorem assures the existence of a conformal map $\Psi\colon\mathbb{B}_1\rightarrow\Omega$.
Then, the Osgood--Carath\'eodory theorem extends $\Psi$ to a homeomorphism of $\overline{\Omega}$ onto $\overline{\mathbb{B}}_1$, denoted by the same symbol $\Psi$.
Since $\partial\Omega$ is analytic, we can extend $\Psi$ conformally to $\mathbb{B}_\tau$ with some $\tau>1$ by the Schwarz reflection principle.
As studied in \cite{katsurada1990asymptotic}, for the MFS, if we arrange the collocation points $\{z_j\}_{j=1}^N\in\partial\Omega$ and $\{\zeta_k\}_{k=1}^N\subset\mathbb{C}\setminus\overline{\Omega}$ by
\begin{alignat}{2}
    z_j&\coloneqq\Psi(\omega^j),&\quad&j=1,2,\ldots,N,\\
    \zeta_k&\coloneqq\Psi(R\omega^k),&\quad&k=1,2,\ldots,N
\end{alignat}
with $R\in(1,\tau)$ and $\omega\coloneqq\exp(2\pi\mathrm{i}/N)$, an approximate solution uniquely exists for sufficiently large $N$, and an approximation error decays exponentially for analytic boundary data.
The essense of the proof is to decompose the basis function into the disk part and the perturbation part as follows:
\begin{equation}
    \log|\Psi(r\mathrm{e}^{\mathrm{i}\theta})-\Psi(R\mathrm{e}^{\mathrm{i}\phi})|
    =\log\left|\frac{\Psi(r\mathrm{e}^{\mathrm{i}\theta})-\Psi(R\mathrm{e}^{\mathrm{i}\phi})}{r\mathrm{e}^{\mathrm{i}\theta}-R\mathrm{e}^{\mathrm{i}\phi}}\right|+\log|r\mathrm{e}^{\mathrm{i}\theta}-R\mathrm{e}^{\mathrm{i}\phi}|.
\end{equation}
Concerning the DSM, the author developed mathematical theory in \cite{sakakibara2016analysis}. 
The similar idea plays an essential role, and the basis function is decomposed as
\begin{equation}
    -\Re\left(\frac{\nu(\Psi(R\mathrm{e}^{\mathrm{i}\phi}))}{\Psi(r\mathrm{e}^{\mathrm{i}\theta})-\Psi(R\mathrm{e}^{\mathrm{i}\phi})}\right)
    =-\Re\left(\frac{\nu(\Psi(R\mathrm{e}^{\mathrm{i}\phi}))}{\Psi(r\mathrm{e}^{\mathrm{i}\theta})-\Psi(R\mathrm{e}^{\mathrm{i}\phi})}-\frac{\nu(R\mathrm{e}^{\mathrm{i}\phi})}{r\mathrm{e}^{\mathrm{i}\theta}-R\mathrm{e}^{\mathrm{i}\phi}}\right)-\Re\left(\frac{\nu(R\mathrm{e}^{\mathrm{i}\phi})}{r\mathrm{e}^{\mathrm{i}\theta}-R\mathrm{e}^{\mathrm{i}\phi}}\right),
\end{equation}
where $\nu(R\mathrm{e}^{\mathrm{i}\phi})$ denotes the unit outward normal vector to the circle $\partial\mathbb{B}_R$ at the point $R\mathrm{e}^{\mathrm{i}\phi}$, whereas $\nu(\Psi(R\mathrm{e}^{\mathrm{i}\phi}))$ does the one to the Jordan curve $\Psi(\partial\mathbb{B}_R)$ at $\Psi(R\mathrm{e}^{\mathrm{i}\phi})$.
The first-term in the right-hand side can be regarded as a good perturbation in some sense.
The second-term in the right-hand side is the basis function for the case of disk, so it can be modified via the DSM-QR.
Hence, it is natural to imagine that the set of functions
\begin{equation}
    \left\{-\Re\left(\frac{\nu(\zeta_k)}{\Psi(r\mathrm{e}^{\mathrm{i}\theta})-\zeta_k}-\frac{\omega^k}{r\mathrm{e}^{\mathrm{i}\theta}-R\omega^k}\right)+\psi_k(r/\rho,\theta)\right\}_{k=1}^N
\end{equation}
would be a new basis function, which removes the ill-conditioning completely.
Although there is no mathematical theory concerning the above idea, we below show several numerical experiments that exemplify the effectiveness of the above strategy.

\begin{example}
    \label{ex:Jordan_1}
    We first show the result where the conformal mapping is given by a polynomial:
    \begin{equation}
        \Psi(z)\coloneqq \frac{4}{5}z+\frac{1}{10}z^5.
    \end{equation}
    We use the same boundary value $f(x,y)=x^2y^3$, and set parameters as $R\coloneqq1.05$, $N=2l+1$, and $l_{\max}\coloneqq500$.
    \begin{figure}[tb]
        \centering
        \includegraphics[width=.7\hsize]{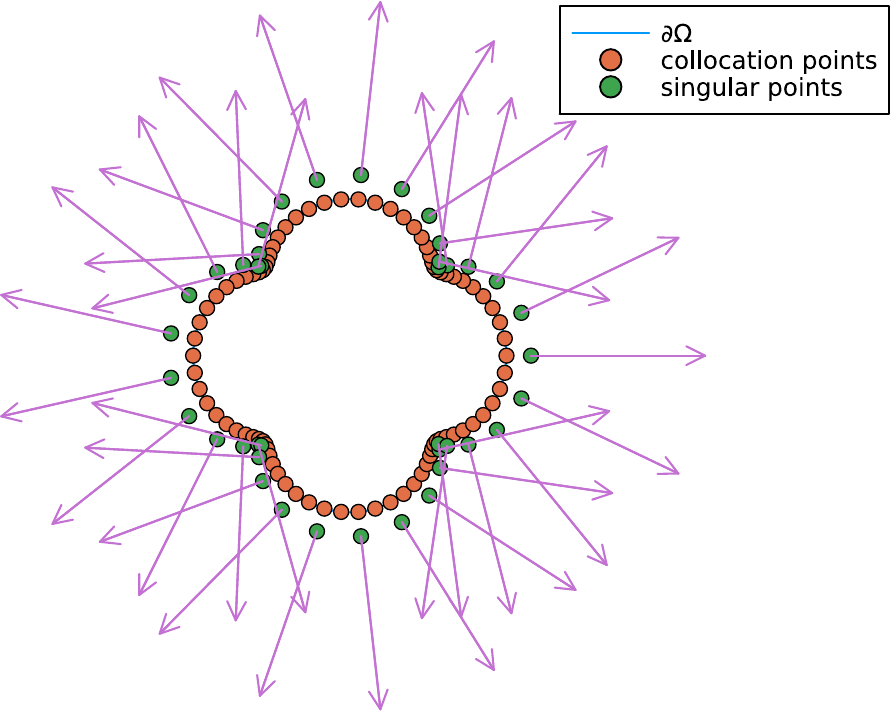}
        \caption{Arrangements of the collocation points, the singular points, and the dipole moments in Example~\ref{ex:Jordan_1}.}
        \label{fig:Jordan_1_collocation_singular}
    \end{figure}
    \begin{figure}[tb]
        \begin{minipage}{.5\hsize}
            \centering
            \includegraphics[width=\hsize]{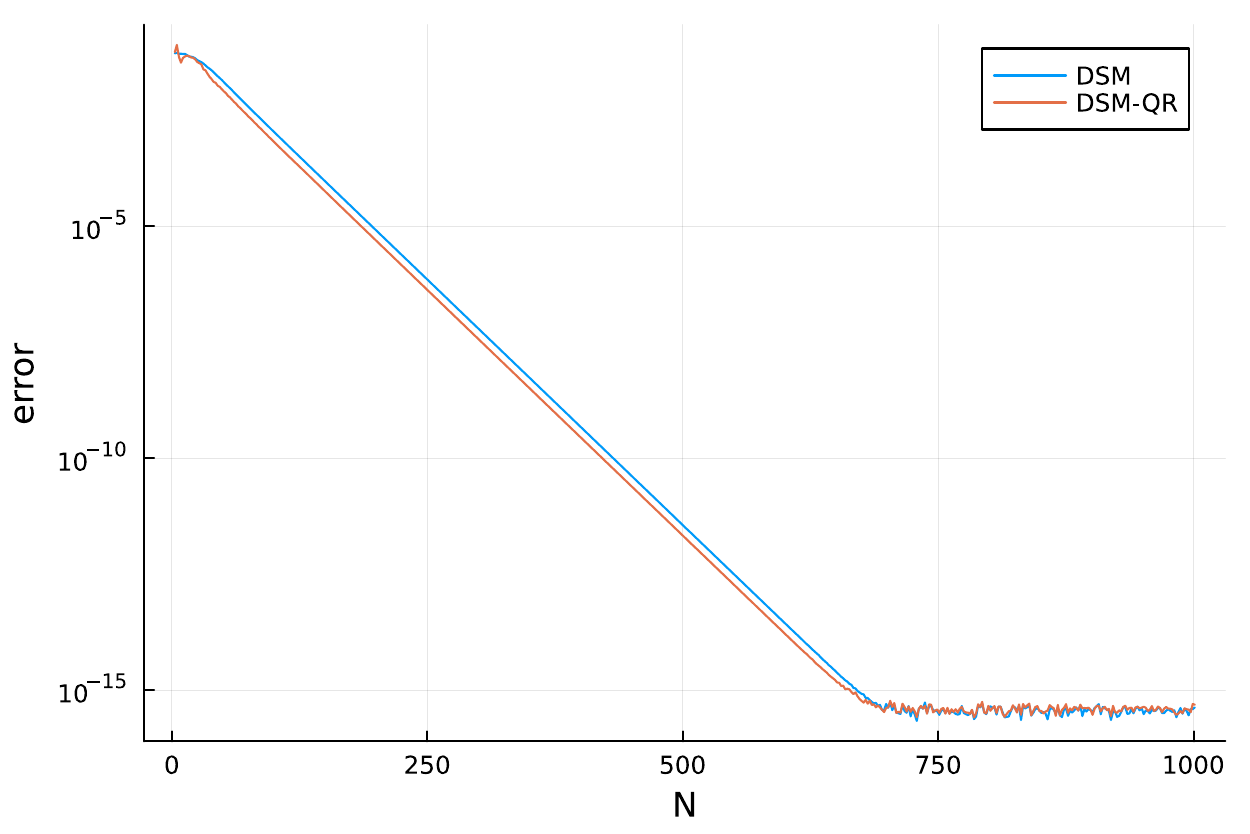}
        \end{minipage}
        \begin{minipage}{.5\hsize}
            \centering
            \includegraphics[width=\hsize]{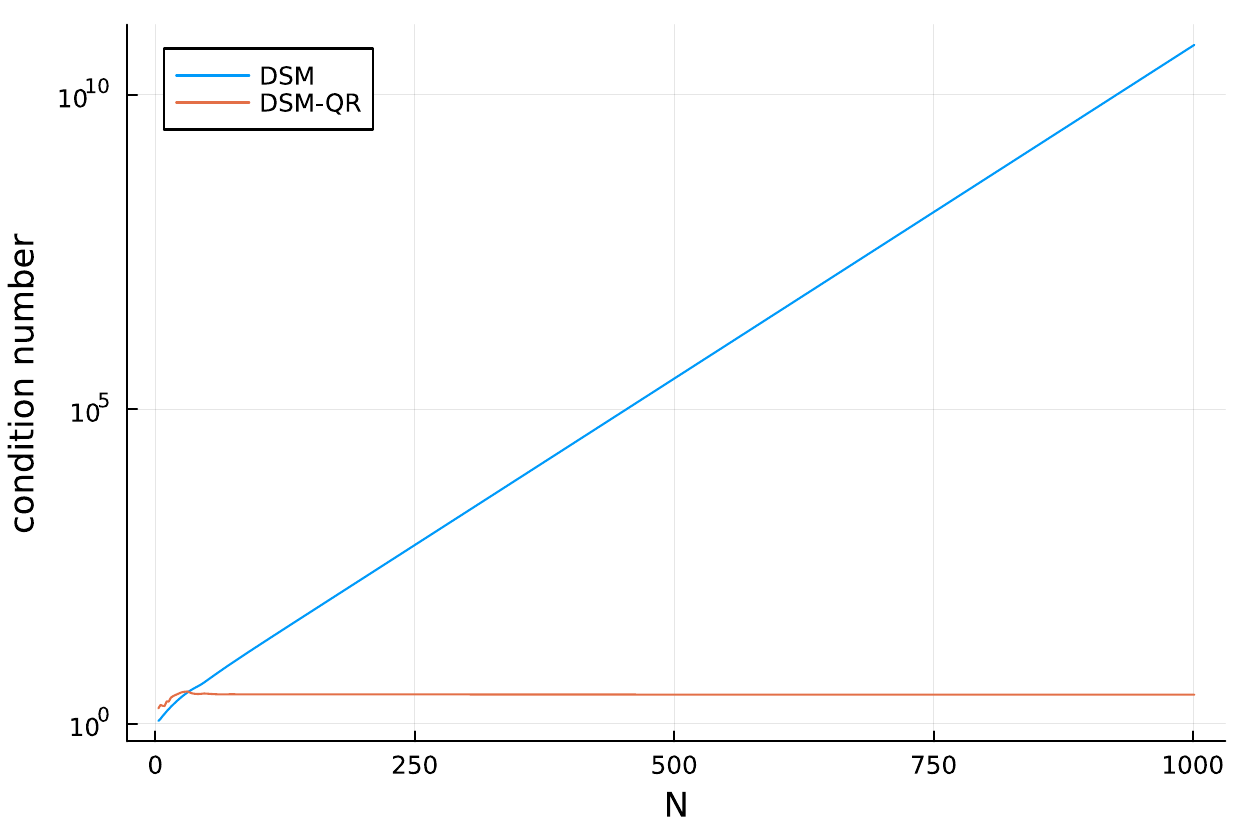}
        \end{minipage}%
        \caption{Plots of the $L^\infty$-norms of errors and the condition numbers in Example~\ref{ex:Jordan_1}.
        (Left) $L^\infty$-error; (Right) condition number.}
        \label{fig:Jordan_1_error_cond-num}
    \end{figure}
    Figure~\ref{fig:Jordan_1_error_cond-num} demonstrates results of numerical experiments.
    As in the case of disk, we can observe that the approximation error decays exponentially for both the DSM and the DSM-QR.
    Concerning the condition number, it increases exponentially for the DSM, while it is of $\mathrm{O}(1)$ for the DSM-QR.
    Namely, the aforementioned strategy works well.
\end{example}

\begin{example}
    \label{ex:Jordan_2}
    To strengthen the validity of the above idea, another numerical experiment will be performed.
    We consider the Joukowski-type transformation, defined as
    \begin{equation}
        \Psi(z)\coloneqq w+\frac{1}{2w},\quad
        w\coloneqq z-\frac{1}{5}+\frac{\mathrm{i}}{5}.
    \end{equation}
    We use the same boundary value $f(x,y)=x^2y^3$, and set parameters as $R\coloneqq1.1$, $N=2l+1$, and $l_{\max}\coloneqq500$.
    \begin{figure}[tb]
        \centering
        \includegraphics[width=.7\hsize]{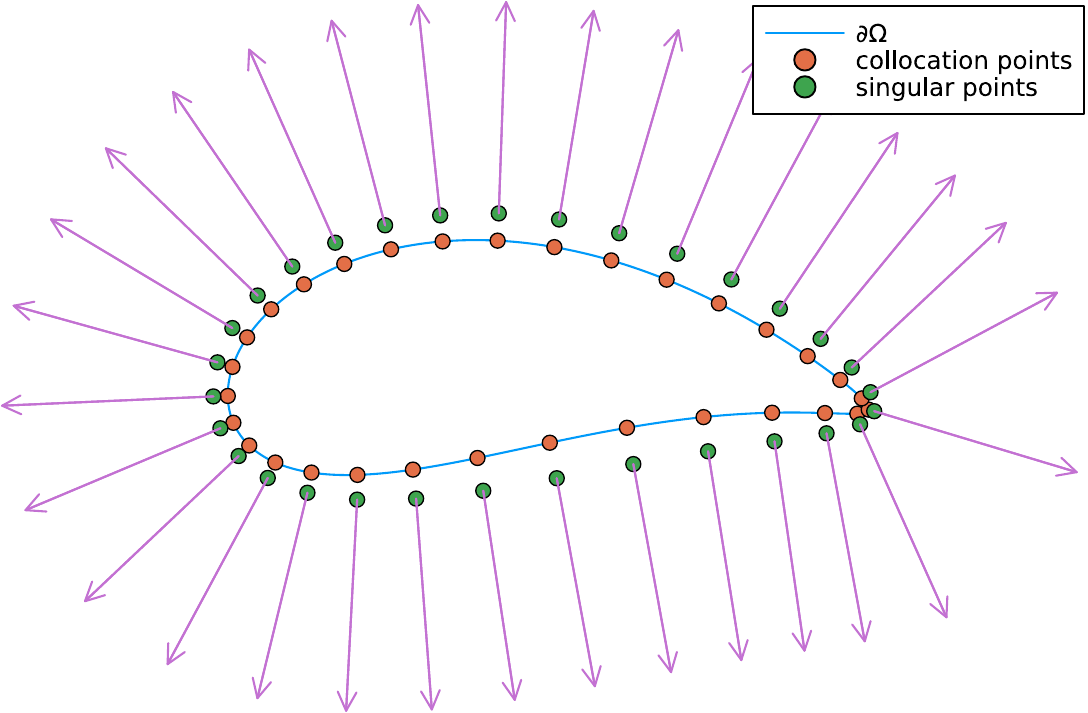}
        \caption{Arrangements of the collocation points, the singular points, and the dipole moments in Example~\ref{ex:Jordan_2}.}
        \label{fig:Jordan_2_collocation_singular}
    \end{figure}
    \begin{figure}[tb]
        \begin{minipage}{.5\hsize}
            \centering
            \includegraphics[width=\hsize]{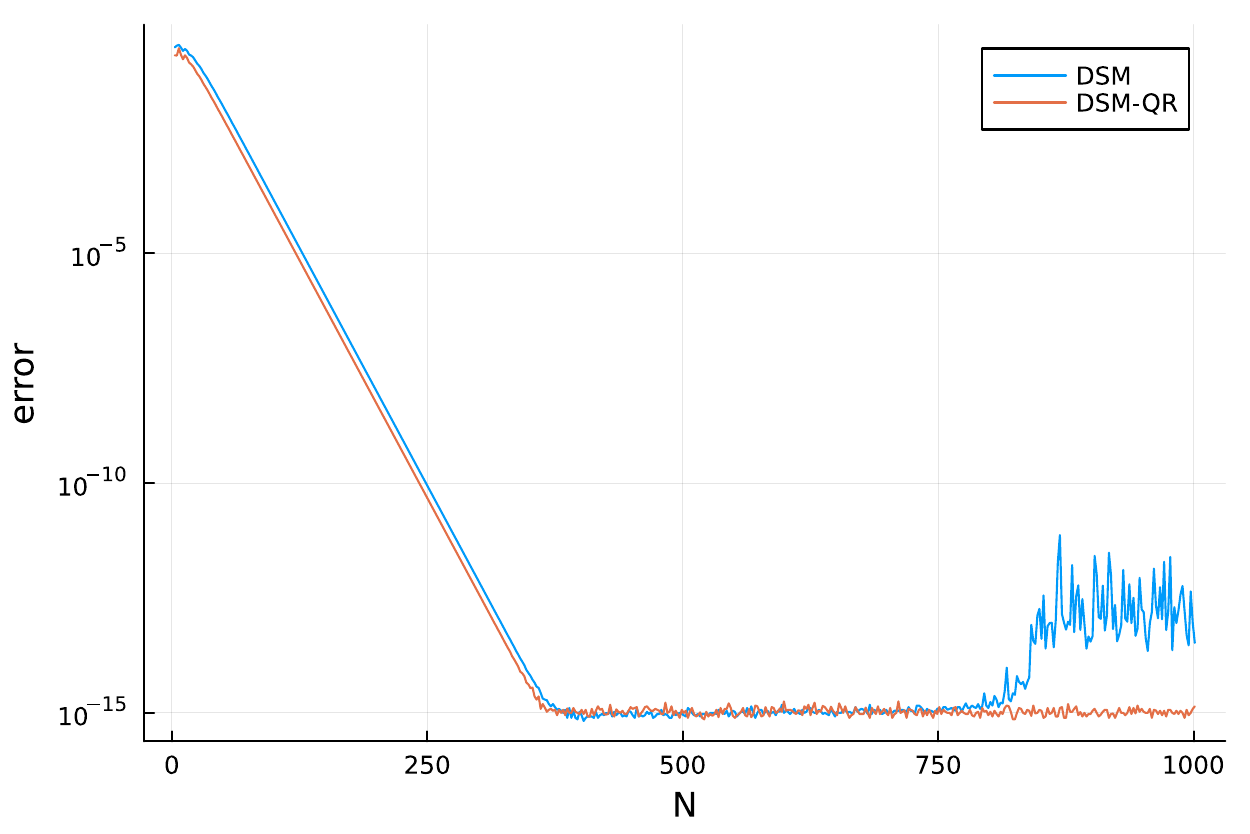}
        \end{minipage}
        \begin{minipage}{.5\hsize}
            \centering
            \includegraphics[width=\hsize]{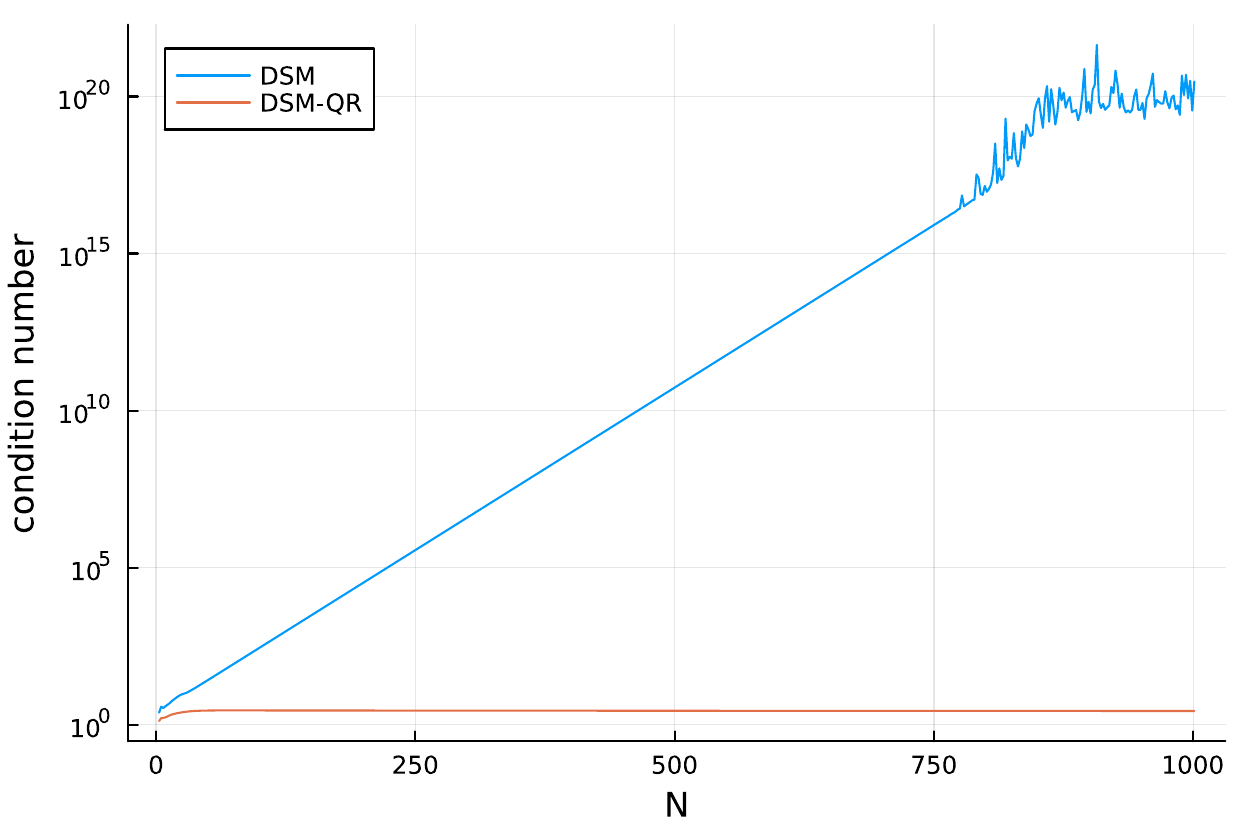}
        \end{minipage}%
        \caption{Plots of the $L^\infty$-norms of errors and the condition numbers in Example~\ref{ex:Jordan_2}.
        (Left) $L^\infty$-error; (Right) condition number.}
        \label{fig:Jordan_2_error_cond-num}
    \end{figure}
    In this case, the curvature of boundary $\partial\Omega$ is very large at the right end, and it is generally difficult to obtain a highly accurate approximate solution. 
    However, as shown in Figure~\ref{fig:Jordan_2_error_cond-num}, DSM and DSM-QR show exponential decay of approximation error. 
    This indicates that DSM and DSM-QR are even more powerful in regions with complex geometry. 
    Moreover, even in this case, the ill-conditionality is completely eliminated in the DSM-QR, which truly demonstrates the high usefulness of the DSM-QR.
\end{example}

\section{Concluding remarks}
\label{sec:summary}

This paper puts forth a new approach, designated as the DSM-QR, which draws inspiration from the MFS-QR. 
It then proceeds to undertake a comprehensive mathematical analysis of this proposed approach within the context of the disk region. 
Theorems~\ref{thm:unique-existence} and \ref{thm:error_estimate} demonstrate that the approximate solution exists uniquely and that the approximation error decays in accordance with the expected order in the context of the MFS. 
A noteworthy outcome is the mathematical proof presented in Theorem~\ref{thm:condition-number}, which demonstrates that the condition number is of $\mathrm{O}(1)$. 
Moreover, despite the absence of a mathematical analysis, we were able to extend the DSM-QR to the general Jordan region through the use of conformal mapping. 
In a previous study~\cite{antunes2022well-conditioned}, a method based on singular value decomposition was shown to be effective in removing ill-conditionality in star-shaped domains. 
However, the main feature of the method presented in this paper is that it is based on conformal mapping, which allows for domain shapes to be unrestricted.

Further research should be conducted in the following areas:
Firstly, the mathematical justification of the extension to Jordan regions should be considered. 
This can be achieved by applying the mathematical analysis methods of the DSM in the Jordan domain~\cite{sakakibara2016analysis}. 
An additional avenue for exploration is the expansion of the problem region in terms of its dimensionality. 
This is a direction that has been pursued by the MFS-QR and the MFS-SVD as well, but thus far, these methods have only demonstrated the ability to remove ill-conditionality in planar regions. 
They have not yet succeeded in extending this approach to domains of greater dimensionality, such as three-dimensional regions. 
Moreover, it would be advantageous to consider extensions to problems in multiply-connected domains, even if they are in two-dimensional regions, for practical applications. 
Finally, since the MFS-QR has been developed for the Helmholtz equation~\cite{antunes2018numerical}, it is crucial to investigate whether the DSM-QR can be extended to other equations as well.

\subsubsection*{Acknowledgements}

This work was supported by JSPS KAKENHI Grant Numbers JP18K13455, JP22K03425, JP22K18677, JP23H00086.

\bibliographystyle{siam}
\bibliography{MFS}

\begin{thebibliography}{10}

\bibitem{antunes2018numerical}
{\sc P.~R.~S. Antunes}, {\em A numerical algorithm to reduce ill-conditioning in meshless methods for the {H}elmholtz equation}, Numer. Algorithms, 79 (2018), pp.~879--897.

\bibitem{antunes2018reducing}
\leavevmode\vrule height 2pt depth -1.6pt width 23pt, {\em Reducing the ill conditioning in the method of fundamental solutions}, Adv. Comput. Math., 44 (2018), pp.~351--365.

\bibitem{antunes2022well-conditioned}
\leavevmode\vrule height 2pt depth -1.6pt width 23pt, {\em A well-conditioned method of fundamental solutions for {L}aplace equation}, Numer. Algorithms, 91 (2022), pp.~1381--1405.

\bibitem{barnett2008stability}
{\sc A.~H. Barnett and T.~Betcke}, {\em Stability and convergence of the method of fundamental solutions for {H}elmholtz problems on analytic domains}, J. Comput. Phys., 227 (2008), pp.~7003--7026.

\bibitem{branden2007discrete}
{\sc H.~Brand\'en, S.~Holmgren, and P.~Sundqvist}, {\em Discrete fundamental solution preconditioning for hyperbolic systems of {PDE}}, J. Sci. Comput., 30 (2007), pp.~35--60.

\bibitem{branden2005preconditioners}
{\sc H.~Brand\'en and P.~Sundqvist}, {\em Preconditioners based on fundamental solutions}, BIT, 45 (2005), pp.~481--494.

\bibitem{ei2022method}
{\sc S.-I. Ei, H.~Ochiai, and Y.~Tanaka}, {\em Method of fundamental solutions for {N}eumann problems of the modified {H}elmholtz equation in disk domains}, J. Comput. Appl. Math., 402 (2022), pp.~Paper No. 113795, 27.

\bibitem{folland1995introduction}
{\sc G.~B. Folland}, {\em Introduction to partial differential equations}, Princeton University Press, Princeton, NJ, second~ed., 1995.

\bibitem{fornberg2007stable}
{\sc B.~Fornberg and C.~Piret}, {\em A stable algorithm for flat radial basis functions on a sphere}, SIAM J. Sci. Comput., 30 (2007/08), pp.~60--80.

\bibitem{katsurada1989mathematical}
{\sc M.~Katsurada}, {\em A mathematical study of the charge simulation method. {II}}, J. Fac. Sci. Univ. Tokyo Sect. IA Math., 36 (1989), pp.~135--162.

\bibitem{katsurada1990asymptotic}
\leavevmode\vrule height 2pt depth -1.6pt width 23pt, {\em Asymptotic error analysis of the charge simulation method in a {J}ordan region with an analytic boundary}, J. Fac. Sci. Univ. Tokyo Sect. IA Math., 37 (1990), pp.~635--657.

\bibitem{katsurada1994charge}
\leavevmode\vrule height 2pt depth -1.6pt width 23pt, {\em Charge simulation method using exterior mapping functions}, Japan J. Indust. Appl. Math., 11 (1994), pp.~47--61.

\bibitem{katsurada1988mathematical}
{\sc M.~Katsurada and H.~Okamoto}, {\em A mathematical study of the charge simulation method. {I}}, J. Fac. Sci. Univ. Tokyo Sect. IA Math., 35 (1988), pp.~507--518.

\bibitem{katsurada1996collocation}
{\sc M.~Katsurada and H.~Okamoto}, {\em The collocation points of the fundamental solution method for the potential problem}, Comput. Math. Appl., 31 (1996), pp.~123--137.

\bibitem{kazashi2014error}
{\sc Y.~Kazashi and M.~Sugihara}, {\em Error estimation for the invariant scheme of charge simulation method on a disc with scattered points}, Jpn. J. Ind. Appl. Math., 31 (2014), pp.~305--321.

\bibitem{li2005convergence}
{\sc X.~Li}, {\em On convergence of the method of fundamental solutions for solving the {D}irichlet problem of {P}oisson's equation}, Adv. Comput. Math., 23 (2005), pp.~265--277.

\bibitem{li2008convergence}
\leavevmode\vrule height 2pt depth -1.6pt width 23pt, {\em Convergence of the method of fundamental solutions for {P}oisson's equation on the unit sphere}, Adv. Comput. Math., 28 (2008), pp.~269--282.

\bibitem{li2009method}
{\sc Z.-C. Li}, {\em The method of fundamental solutions for annular shaped domains}, J. Comput. Appl. Math., 228 (2009), pp.~355--372.

\bibitem{murota1995comparison}
{\sc K.~Murota}, {\em Comparison of conventional and ``invariant'' schemes of fundamental solutions method for annular domains}, Japan J. Indust. Appl. Math., 12 (1995), pp.~61--85.

\bibitem{sakakibara2016analysis}
{\sc K.~Sakakibara}, {\em Analysis of the dipole simulation method for two-dimensional {D}irichlet problems in {J}ordan regions with analytic boundaries}, BIT, 56 (2016), pp.~1369--1400.

\bibitem{sakakibara2017asymptotic}
\leavevmode\vrule height 2pt depth -1.6pt width 23pt, {\em Asymptotic analysis of the conventional and invariant schemes for the method of fundamental solutions applied to potential problems in doubly-connected regions}, Jpn. J. Ind. Appl. Math., 34 (2017), pp.~177--228.

\bibitem{sakakibara2017method}
\leavevmode\vrule height 2pt depth -1.6pt width 23pt, {\em Method of fundamental solutions for biharmonic equation based on {A}lmansi-type decomposition}, Appl. Math., 62 (2017), pp.~297--317.

\bibitem{smyrlis2006method}
{\sc Y.-S. Smyrlis}, {\em The method of fundamental solutions: a weighted least-squares approach}, BIT, 46 (2006), pp.~163--194.

\end{thebibliography}

\end{document}